\documentclass[a4paper, 12pt, oneside, reqno, notitlepage]{amsart}
\usepackage{amsmath,amssymb,amsthm,graphicx,mathrsfs,bbm,url}
\usepackage[margin=3cm]{geometry}
\usepackage{enumitem}
\usepackage{mathtools}
\usepackage[usenames,dvipsnames]{color}
\usepackage[colorlinks=true,linkcolor=red,citecolor=green]{hyperref}
\usepackage[super]{nth}
\usepackage[open, openlevel=2, depth=3, atend]{bookmark}
\hypersetup{pdfstartview=XYZ}
\usepackage[font=footnotesize]{caption}
\usepackage{mathabx}
\usepackage[colorinlistoftodos]{todonotes}
\usepackage{tikz-cd}
\tikzcdset{row sep/normal=3.6em,column sep/normal=3.6em}

\usepackage{epstopdf}

\theoremstyle{plain}
\newtheorem{theorem}{Theorem}[section]
\newtheorem{lemma}[theorem]{Lemma}

\newtheorem{corollary}[theorem]{Corollary}

\theoremstyle{definition}
\newtheorem{definition}[theorem]{Definition}

\theoremstyle{remark}
\newtheorem{remark}[theorem]{Remark}

\numberwithin{equation}{section}

\title
[Smooth invariant foliations for semiflows] 
{Smooth invariant foliations and Koopman eigenfunctions about stable equilibria of semiflows}

\author{Gergely Buza}
\address{Department of Applied Mathematics and Theoretical Physics, University of
Cambridge, Cambridge CB3 0WB, UK}
\email{gb643@cam.ac.uk}

\begin{document}

\maketitle

\begin{abstract}
    We consider a $C^r$ semiflow $\{ \varphi_t \}_{t \geq 0}$ on a Banach space $X$ admitting a stable fixed point $x$.
    We show, along the lines of the parameterization method \cite{Cabre2003a}, the existence of a $C^r$ invariant foliation tangent to $X_1$ at $x$, for an arbitrary $D \varphi_t(x)$-invariant subspace $X_1 \subset X$ satisfying some additional spectral conditions.
    Uniqueness ensues in a subclass of sufficiently smooth invariant foliations tangent to $X_1$ at $x$.
    We then draw relations to Koopman theory, and thereby establish the existence and uniqueness, in some appropriate sense, of $C^r$ Koopman eigenfunctions.
    We demonstrate that these results apply to the case of the Navier-Stokes system, the archetypal example considered by the modern upheaval of applied 'Koopmanism'.
\end{abstract}

\section{Introduction}

Given a dynamical system on a manifold, an invariant foliation is, roughly speaking, a collection of leaves tessellating the manifold (or a distinguished region of it) such that leaves are mapped into one another under the action of the dynamical system.
Structures as such crop up naturally 
and have been the subject of study 
for a long time
throughout the development of dynamical systems theory. 
Perhaps the most prominent 
examples are the stable/untable foliations of a normally hyperbolic invariant manifold, leaves of which are characterized by sharp forward/backward asymptoticity and their tangency to the stable/unstable subbundles \cite{HirschPughShub70,fenichel1971persistence} (see also \cite{anosov1967geodesic,roberts1989appropriate,lu1991hartman,chen1997invariant,eldering2018global,szalai2020invariant}, to name a few alternative examples).

The present paper is concerned with foliations that are instead tangent at a stable 
fixed point to arbitrary subspaces invariant under the linearized dynamics.
Considerations as such 
have been explored recently in finite dimensions, so as to provide practical reduced-order models to the field of engineering mechanics \cite{szalai2020invariant,szalai2023data}. 
Here, we generalize these results to semiflows on Banach spaces. 
The main objective in mind is to conform to the setting of semilinear parabolic evolution equations \cite{henry1981geometric},
and hence to that of the Navier-Stokes system, but the scope of applications is by no means limited to this.

This line of work was motivated  partially by the 
desideratum
to tackle 
the existence/uniqueness problem of Koopman eigenfunctions \cite{koopman1931hamiltonian,mezic2005spectral,mezic2020spectrum}, 
which 
are closely related to foliations\footnote{The associated notion in Koopman theory is that of \textit{isostables} \cite{mauroy2013isostables}.}. 
The Koopman operator, given by the pullback of observable functions along the dynamics, 
is an almost century-old concept \cite{koopman1931hamiltonian} that has seen an increase in popularity upon being reignited semi-recently \cite{mezic2004comparison,mezic2005spectral}.
Its use through dynamic mode decomposition (DMD)--a tool that aims to approximate Koopman eigenfunctions--has since become an integral part of the fluid dynamics literature, and is routinely employed in model reduction, post-processing and related areas \cite{schmid2010dynamic,schmid2022dynamic,schmid2011applications}.
However, 
 efforts to ascertain the conditions under which these structures exist have lagged behind. 
While Sternberg- and Poincaré-type linearization theorems \cite{sternberg1957local,poincare1879proprietes} had been  
long known to produce Koopman eigenfunctions \cite{mauroy2013isostables,mezic2020spectrum}, the first attempt to identify the necessary set of assumptions and answer the uniqueness question properly was the 
work of \cite{kvalheim2021existence}. 
This latter proceeding addresses the finite-dimensional case wholly, but, 
being that the main realm of applications is the infinite-dimensional setting of Navier-Stokes, we figured the results would be worth extending to this case.

The issue becomes particularly intricate in infinite dimensions, since one may no longer 
fall back on full-fledged linearization theorems (of Sternberg/Poincaré) in the absence of 
meticulous results,
as the set of spectral nonresonance conditions becomes near impossible to verify (or, in most cases, they simply fail).
Via considering foliations, the number of nonresonance assumptions required are relaxed in two ways: we merely need a semiconjugacy in place of a conjugacy, and it is no longer necessary to linearize the reduced dynamics. 
Partial and full linearization results are then obtained as special cases (which already exist for the case of maps even in infinite dimensions, see \cite{elbialy2001local}).
With the number of spectral conditions reduced to a feasible amount, 
rigorous algorithms designed to approximate the spectrum 
may serve to verify them \cite{boffi2000problem,colbrook2022foundations},
thereby confirming the existence of foliations and Koopman eigenfunctions on an example-by-example basis. 
This, in turn, could provide insurance to the practitioner that results obtained by numerical procedures such as DMD do pertain to the full underlying system.

The core of the proof herein proceeds 
in line with that of the parameterization method \cite{Cabre2003a,cabre2003parameterization2,cabre2005parameterization3,haro2016parameterization} (and hence also with those of \cite{sternberg1957local,szalai2020invariant,kvalheim2021existence}).
Alternative routes to establishing foliations exist, and have been employed successfully even in the infinite-dimensional, semiflow setting \cite{chow1991smooth,chen1997invariant,bates2000invariant}. 
These procedures
rely crucially on the foliation's  
defining invariant subspace being associated to a half-space of the spectrum--so that Lyapunov-Perron, graph transform or Irwin-type proofs are applicable--and hence are not quite adequate for the purposes herein.
Moreover, the leaves of these foliations depend only continuously on the base point -- yielding projection maps that are merely continuous.
In contrast, the approach herein will yield smoother base-point dependence, which will be crucial in establishing the existence of \textit{principal} Koopman eigenfunctions (which, in some sense, are the only useful ones).
To compare the assumptions with \cite{chow1991smooth,chen1997invariant}, we note the following points.
First, spectral (and Lyapunov) stability is assumed herein.
This means that we will not be able to produce smooth Koopman eigenfunctions about unstable equilibria; and to produce $C^0$ ones, we resort to the pseudostable foliations of \cite{chen1997invariant} (obtained similarly to Irwin's method \cite{Irwin72,irwin1980new}). 
In the context of the defining invariant subspaces permitted by \cite{chow1991smooth,chen1997invariant}, stability is the only additional assumption we make.
To encompass more general subspaces, which do not necessarily satisfy the exponential dichotomy assumptions of \cite{chow1991smooth,chen1997invariant}, we require in addition to stability some nonresonance conditions described in Section~\ref{sect:statement} below.

\subsection{Preliminaries, overview and organization} 
\label{sect:intro_tech}

We begin by introducing some terminology.

Let $M$ be a Banach manifold\footnote{Proceeding forward, all manifolds will be assumed at least Banach.}.
A family of maps $\{\varphi_t\}_{t \geq 0}$, $\varphi_t:M \to M$, is called a \textit{semiflow} on $M$ if
\begin{subequations} \label{eq:sf}
    \begin{align}
    \varphi_0 &= \mathrm{id}_M, \label{eq:sf1}\\ 
    \varphi_t \circ \varphi_s &= \varphi_{t+s}, \qquad \text{for all } t,s \geq 0. \label{eq:sf2}
\end{align}\end{subequations}
If \eqref{eq:sf2} holds for all $t,s \in \mathbb{R}$, $\varphi$ is called a \textit{flow}.
A \textit{local semiflow} is a map $\varphi:\mathcal{D}^\varphi \to M$ satisfying \eqref{eq:sf} wherever defined, where  $\mathcal{D}^\varphi \subset \mathbb{R}^{\geq 0} \times M$ is an open subset such that $(0,p) \in \mathcal{D}^\varphi$ for all $p \in M$.
Let $\mathcal{D}_t^\varphi : = \{p \in M \, | \, (t,p) \in \mathcal{D}^\varphi \}$ and $\mathcal{D}_p^\varphi : = \{t \in \mathbb{R}^{\geq 0} \, | \, (t,p) \in \mathcal{D}^\varphi \}$.

We proceed by introducing foliations formally,
through the simplest possible route, following \cite{lawson1974foliations}.\footnote{We note that alternative definitions exist that permit the individual leaves to be smoother than the foliation itself, usually termed $C^k \times C^\ell$ foliations with $k \geq \ell$ (see, e.g., \cite{burchard1992smooth}). These coincide with Definition~\ref{def:foliation} if $k = \ell$, and we shall make no use of this additional degree of freedom herein.}

\begin{definition}[Foliations] \label{def:foliation} 
Suppose $M$ is a manifold modeled on $X$, for $X$ a Banach space admitting a splitting $X = X_0 \oplus X_1$ into closed linear subspaces $X_0$ and $X_1$. 
By an \textit{$X_1$-modeled, $C^r$ foliation} of $M$, $r \in \mathbb{N}_0 \cup \{ \infty,\omega \}$, 
we mean a decomposition of $M$ into a union of disjoint connected subsets
$\{ \mathscr{L}_i \}_{i \in I}$, called \textit{leaves} of the foliation,
such that there exists a $C^r$ atlas on $M$ composed of charts of the form $h : M \supset O \to X$ that map leaves passing through their domains $\mathscr{L}_i \cap O$ to sets of the form $(\{x_0\} \times X_1) \cap \mathrm{im}(h)$ (with $x_0$ ranging through $X_0$ as $i \in I$ is varied). 
\end{definition}

Note that, by this definition, the leaves $\mathscr{L}_i$ are submanifolds of $M$ modeled on $X_1$. 

It is often preferable to parameterize the leaves $\mathscr{L}_i$ 
by the points they contain instead.
In this context, two leaves $\mathscr{L}_p$ and $\mathscr{L}_q$, for $p,q \in M$, are either disjoint or identical.

\begin{definition}[(Locally) invariant foliations] \label{def:foliation_invariance}
Given a semiflow $\{ \varphi_t \}_{t \geq 0}$ on $M$, a foliation $\{ \mathscr{L}_p \}_{p \in M}$
is said to be \textit{(positively)\footnote{In this work, we only deal with positively invariant foliations and hence this prefix will generally be omitted.} invariant} if $\varphi_t (\mathscr{L}_p) \subset \mathscr{L}_{\varphi_t (p)}$, for each $p \in M$ and $t \geq  0$.

We say it is a \textit{(positively) invariant foliation about $p \in M$} if the above invariance property holds for an open neighborhood $O \subset M$ of $p$, i.e., $\varphi_t (O \cap \mathscr{L}_q ) \subset \mathscr{L}_{\varphi_t(q)}$ for all $q \in O$ and $t \geq 0$.

We say it is \textit{locally (positively) invariant about $p \in M$} if there exists an open neighborhood of $p$, $O \subset M$, such that $\varphi_t ( \bigcap_{s \in [0,t]} \varphi_s^{-1}(O) \cap \mathscr{L}_q)  \subset \mathscr{L}_{\varphi_t (q)} $ for those $q \in O$ and $t \geq 0$ with $\varphi_s(q) \in O$ for all $s \in [0,t]$. 
\end{definition}

Locally, any $C^r$ foliation with $r \geq 1$ may be viewed as a submersion. 
\begin{definition}[Submersions] \label{def:submersion}
    Let $M$ and $N$ be two $C^r$ manifolds, $r \geq 1$. 
    A $C^r$ map $\pi:M \to N$ is said to be a submersion at $p \in M$ if $D \pi(p) : T_p M \to T_{\pi(p)}N$ is surjective and its kernel splits (i.e., $\ker ( D \pi(p) )$ is complemented in $T_p M$).
    It is called a submersion if it is a submersion at all $p \in M$.
\end{definition}

The local equivalence of submersions and foliations follows from an immediate corollary of the inverse function theorem (Corollary~5.8, \cite{lang2012fundamentals}).
Suppose $\pi:M \to N$ is a $C^r$ submersion at $p \in M$, as in Definition~\ref{def:submersion}, with $M$ being modeled on $X$, a Banach space.
Take charts $\sigma$ and $\tau$ on neighborhoods of $p$ and $\pi(p)$, respectively, such that $\sigma(p) = 0$.
Then, $X_1 : = \ker ( D(\tau \circ \pi \circ \sigma^{-1})(0) )$ is complemented in $X$, denote any choice as such by $X_0$.
According to the splitting $X = X_0 \oplus X_1$, the map
\begin{equation}
    (x_0,x_1) \mapsto \left( (D(\tau \circ \pi \circ \sigma^{-1}) \vert_{X_0} (0))^{-1} \tau \circ \pi \circ \sigma^{-1}(x_0,x_1),x_1 \right)
    \label{eq:subm_2_foli}
\end{equation}
is a local $C^r$ diffeomorphism about $0$ by the inverse function theorem.
Precomposing \eqref{eq:subm_2_foli} with $\sigma$ yields a chart about $p$ of the form required by Definition~\ref{def:foliation}. 
The converse is obvious, as the composition $ M \supset O \xrightarrow{h} X_0 \times X_1 \xrightarrow{\mathrm{proj}} X_0$ is a submersion on $O$ for a chart $h$ defined on $O$ as in Definition~\ref{def:foliation}.

The particular choice of the target space $X_0$ for the submersion
holds no significance and is defined only up to a diffeomorphism; indeed, any (sufficiently smooth) manifold modeled on $X_0$ and thus any choice of subspace complementing $X_1$ would suffice,
for these choices do not alter the resulting foliation.

In the realm of submersions $\pi:M \to N$, the invariance property may be written $\varphi_t (\pi^{-1} (\pi(p))) \subset \pi^{-1}(\pi \circ \varphi_t (p))$, for $p \in M$ and $t \geq 0$,
which is equivalent to the existence of a semiflow $\{ \vartheta_t \}_{t \geq 0}$ on $N$ such that $\pi \circ \varphi_t = \vartheta_t \circ \pi $ for all $t \geq 0$.

As previously outlined in the introduction, the present work is concerned with foliations locally defined about (stable) equilibria, and hence our exposition will be focused on the case $M = U$, for $U$ some open subset of a Banach space containing $0$, which we assume without loss of generality to be the fixed point. 
Note, however, that since the conclusions reached herein are purely local, they continue to hold on sufficiently smooth manifolds (locally, about a fixed point) by passing through a chart.

The objective of this paper 
is thus reduced to the following.
We seek submersions $\pi : X \supset U \to X_0$ at $0 \in X$ of higher regularity, on a Banach space $X$, where $X_0$ is any complement to $X_1 := \ker ( D\pi(0) )$, satisfying the invariance relation
\begin{equation}
     \pi \circ \varphi_t = \vartheta_t \circ \pi, \qquad t \geq 0,
    \label{eq:intro_invariance}
\end{equation}
for some (local) semiflow $\{ \vartheta_t \}_{t \geq 0}$ on $X_0$.

Distinguishing features of (smooth) invariant foliations are their tangency to some closed linear subspace $X_1$, by which we mean the tangency of $\mathscr{L}_0$ to $X_1$ at $0$.
What we end up showing is that, for any closed linear subspace $X_1$ invariant under $D \varphi_t (0)$ satisfying some additional spectral conditions (which depend crucially on the choice of $X_1$), there exists an invariant foliation about $0$ tangent to $X_1$.
Moreover, said foliation is unique among sufficiently smooth invariant foliations tangent to $X_1$.
We remark that the $D \varphi_t (0)$-invariance of $X_1$ is forced by \eqref{eq:intro_invariance} (see also Remark~\ref{remark:linearpart}),
but $X_1$ need not necessarily be a spectral subspace, nor does it have to have an invariant complement.

In many scenarios, particularly in engineering and related areas, it is desirable to simplify $\vartheta$ (in the sense of normal form theory \cite{murdock2003normal}).
The optimal end result is if $\vartheta_t$ can be chosen a linear map for all $t \geq 0$; in the language of Definition~\ref{def:foliation}, this amounts to the existence of a (local) $C^r$ change of coordinates according to which the dynamics take the form $(T_t^0,\widetilde{\varphi}_t^1)$, for a linear semigroup $T_t^0$ acting on $X_0$ and $\widetilde{\varphi}_t^1$, a semiflow on $X_1$.
These are usually referred to as \textit{linearizing semiconjugacies} and are the subject of discussion in a number of recent texts \cite{elbialy2001local,eldering2018global,kvalheim2021existence}. 
We obtain them as a byproduct of the foliation results, under a set of additional 'internal' nonresonance conditions (to be made explicit later) that enable $C^r$ linearization of $\vartheta_t$.

In the following section, we make the above outlined statements precise,
and give the main foliation-related results. 
Their consequences pertaining to Koopman theory are drawn in Section~\ref{sect:koopman}.
In Section~\ref{sect:stablecase}, we reproduce results of \cite{kvalheim2021existence} concerning the existence and uniqueness of Koopman eigenfunctions about stable equilibria,
tailored to the setting of parabolic equations on bounded domains.
Section~\ref{sect:unstablecase} is a slight detour, during the course of which we apply the invariant foliation theorem of \cite{chen1997invariant} to obtain $C^0$ Koopman eigenfunctions about (not necessarily stable) hyperbolic fixed points.
Thereafter, in Section~\ref{sect:example}, we consider the example of fluid flow on a bounded domain, and confirm the applicability of all prior results in the Navier-Stokes setting.
Finally, Section~\ref{sect:proof} contains the proofs of all statements proposed in Section~\ref{sect:statement}.

\section{Statement of the results}
\label{sect:statement}

As is standard procedure in dynamical systems, we first present the main theorem for the case of maps. 
Once this is proven, its uniqueness statement will serve to transition the result to the setting of semiflows.

First, let us introduce some further notational conventions. 
For two Banach spaces $X$ and $Y$, $r \in \mathbb{N}_0$, and an open subset $O \subset X$, we consider the Banach space
\begin{multline*}
    C^r_b(O;Y) = \Big\{ f : O \to Y \; \Big| \; f \text{ is $r$-times differentiable}, \\ x \mapsto D^i f (x) \text{ is bounded and continuous} , \, 0 \leq i \leq r \Big\}
\end{multline*}
with norm 
\begin{displaymath}
    \Vert f \Vert_{C^r_b(O;Y)} = \sup \big\{ \Vert D^i f (x) \Vert \; \big| \; x \in O, \, 0 \leq i \leq r \big\},
\end{displaymath}
where $D$ stands for the Fréchet derivative.
The space $C^\infty_b (O;Y)$ consists of functions $f$ such that $f \in C^r_b (O;Y)$ for all $r \in \mathbb{N}$.
We also consider the space $C^\omega_b (O;Y)$ of bounded complex analytic functions equipped with the supremum norm, whenever $X$ is complex.
If $O$ and $Y$ are obvious from the context, we suppress them from the notation.
We will also continue to refer to functions as being $C^r$ to indicate mere $r$-times continuous differentiability.

For a linear operator $A$ on a complex Banach space $X$, we denote by $\sigma(A)$ its spectrum. 
If $X$ is a real Banach space, then $\sigma(A)$ denotes the spectrum of the complexification of $A$.
We use $\mathrm{P}\sigma(A)$, $\mathrm{R}\sigma(A)$ and $\mathrm{C}\sigma(A)$ to denote the point, residual and continuous spectra, respectively.

Given two sets $\Lambda, \Gamma \subset \mathbb{C}$, we use the notation
\begin{displaymath}
    \Lambda \Gamma = \left\{ \lambda \gamma \; \middle\vert \; \lambda \in \Lambda, \, \gamma \in \Gamma \right\}, \qquad \text{and} \qquad \Lambda + \Gamma = \left\{ \lambda + \gamma \; \middle\vert \; \lambda \in \Lambda, \, \gamma \in \Gamma \right\},
\end{displaymath}
and similarly, $\Lambda^n = \Lambda \cdots \Lambda$ and $n \Lambda = \Lambda + \cdots + \Lambda$ (repeated $n$ times) for $n \in \mathbb{N}$.

We denote the open ball of radius $\varepsilon > 0$ about $x \in X$ by $B_\varepsilon^X(x)$.
If $x = 0$ is the origin, the base point is omitted from the notation.

The hypotheses required for the map version of the main result are as follows.

\begin{enumerate}[label =(H.\arabic*)] 

\item \label{hyp1:X} Let $(X,|\cdot|)$ be a real or complex Banach space admitting a decomposition into closed linear subspaces $X = X_0 \oplus X_1$; denote by $\imath_i: X_i \xhookrightarrow{} X$ the inclusions, by $P_i:X \to X$ the projections with $\ker (P_1) = X_0$, $\ker (P_0) = X_1$, and by $\widetilde{P}_i:X \to X_i$ the projections with restricted range for $i=0,1$. 
(Here, $X_0$ and $X_1$ are endowed with the subspace topology.)

\item \label{hyp2:f} Let $f:O \to X$ be $C^r_b$ smooth (with $O \subset X$ an open neighborhood of $0$), $r \in \mathbb{N} \cup \{\omega,\infty \}$\footnote{When $X$ is a complex Banach space, we assume $r = \omega$, i.e., the complex analytic case (here and everywhere else).}, 
with $f(0) = 0$ and such that $A: = Df(0)$ takes the following form with respect to the splitting $X = X_0 \oplus X_1$ of \ref{hyp1:X}:
\begin{equation}
    A = \begin{pmatrix}
        A_0 & 0 \\
        B & A_1
    \end{pmatrix}.
    \label{eq:A}
\end{equation}
Assume moreover that $0 \notin \sigma(A_0)$, $\sigma(A) \subset B_1^{\mathbb{C}}$; and denote by $\phi = f - A$.

\item \label{hyp3:ell} Let $1 \leq \ell <r$\footnote{For the possible values of $r$, the usual ordering of $\mathbb{N}$ is extended such that $\mathbb{N} < \infty < \omega$ (here and everywhere else).} be an integer such that
\begin{equation}
    \sigma( A_0^{-1}) \sigma(A)^{\ell + 1} \subset B_1^{\mathbb{C}}, 
    \label{eq:ell}
\end{equation}
with $A$ as in \eqref{eq:A}.

\item \label{hyp4:spec1} The spectrum of $A$ from \eqref{eq:A} satisfies 
\begin{displaymath}
    \sigma(A_0)^{n-j} \sigma(A_1)^j \cap \sigma(A_0) = \emptyset
\end{displaymath}
for all pairs $(j,n)$ of integers such that $2 \leq n \leq \ell$ and $1 \leq j \leq n$, with $\ell$ as in \ref{hyp3:ell}.

\item \label{hyp5:spec2} The spectrum of $A$ from \eqref{eq:A} satisfies 
\begin{displaymath}
    \sigma(A_0)^{n}  \cap \sigma(A_0) = \emptyset
\end{displaymath}
for all $m \leq n \leq \ell$ for some integer $m \geq 2$ and $\ell$ as in \ref{hyp3:ell}.

\end{enumerate}

The final assumption, \ref{hyp5:spec2}, is an optional one required only if the reduced dynamics is to be simplified in polynomial order (see below), and is precisely what differentiates foliations from linearizing semiconjugacies.

\begin{theorem} \label{thm:maps}
    Under hypotheses \ref{hyp1:X}-\ref{hyp4:spec1}, the following assertions hold.
    \begin{enumerate}[label =({\roman*})]
    \item \label{thm:maps_1} There exists an open neighborhood of $0$, $U \subset O$, a $C^r_b$ map $\pi : U \to X_0$ with 
    \begin{subequations}\label{eq:map_foliation_diag}
    \begin{equation}
        \pi(0) = 0, \qquad D\pi(0) \vert_{X_0} \in \mathrm{Aut}(X_0),
        \label{eq:Dpi_Aut}
    \end{equation} 
    and a map $g:X_0 \supset O_0 \to X_0$ 
    satisfying
    \begin{equation} \label{eq:invariance_maps}
    \begin{tikzcd} 
    U \arrow[r, "f"] \arrow[d, "\pi"']
    & U \arrow[d, "\pi"] \\
    O_0 \arrow[r, "g"']
    &  O_0
    \end{tikzcd}
    \end{equation}\end{subequations}
    Moreover, $g$ can be chosen to be a polynomial of degree not larger than $\ell$.
    If \ref{hyp5:spec2} holds with some $m \leq \ell$, $g$ can be chosen degree-$(m-1)$ (linear if $m=2$). 
    \item \label{thm:maps_2}
    The solution constructed in \ref{thm:maps_1} is unique among $C^{\ell+1}_b$ solutions in the following sense. 
    Given two pairs of $C^k_b$, $\ell + 1 \leq k \leq r$\footnote{If $r = \omega$, we only consider $k = \omega$, and the norms $\Vert \cdot \Vert_{C^{\ell+1}_b}$ towards the end of the statement should be replaced by $\Vert \cdot \Vert_{C^{\omega}_b}$.}, solutions to \eqref{eq:map_foliation_diag} (over any two neighborhoods in place of $U$), $(g,\pi)$ and $(\tilde{g},\tilde{\pi})$ (with all four functions in $C^{k}_b$), there exists a 
    unique 
    local $C^{k}_b$ diffeomorphism $\theta$ on $X_0$ such that 
    \begin{subequations} \label{eq:uniqueness}
    \begin{align}
         &\tilde{\pi}  = \theta \circ \pi,  & &\text{on } W, \label{eq:pi_uniqueness} \\
         &\tilde{g} = \theta \circ g \circ \theta^{-1},  & &\text{on } \tilde{\pi}(W). \label{eq:g_uniqueness}
    \end{align}\end{subequations}
    The neighborhood $W$ on which \eqref{eq:pi_uniqueness} holds is determined by
    $\Vert f-A \Vert_{C^{\ell+1}_b}$, $\Vert (D \pi(0)\vert_{X_0})^{-1} g \circ D \pi(0)\vert_{X_0} -A_0 \Vert_{C^{\ell+1}_b}$, $\Vert (D \tilde{\pi}(0)\vert_{X_0})^{-1} \tilde{g} \circ D \tilde{\pi}(0)\vert_{X_0} -A_0 \Vert_{C^{\ell+1}_b}$, and
    $\Vert (D \tilde{\pi}(0)\vert_{X_0})^{-1} \tilde{\pi} \vert_{U_0} \circ ( \pi \vert_{U_0} )^{-1 } \circ D \pi(0)\vert_{X_0}- \mathrm{id}_{X_0} \Vert_{C^{\ell+1}_b}$, for an open set $U_0 \subset X_0$ on which both  $\pi$ and $\tilde{\pi}$ are invertible. 
    
    If $g$ and $\tilde{g}$ are both linear, then $\theta = D\tilde{\pi}(0) \vert_{X_0} (D \pi(0)\vert_{X_0})^{-1} \in \mathrm{Aut}(X_0)$. 
    \end{enumerate}
\end{theorem}

\begin{proof}
    The proof is deferred to Section~\ref{sect:proof}.
\end{proof}

\begin{remark} \label{remark:diffeo}
    Given a (local) $C^r_b$ diffeomorphism $\theta$ on $X_0$ and a $C^r_b$ solution pair $(g,\pi)$ of \eqref{eq:map_foliation_diag}, it is clear that $(\theta \circ g \circ \theta^{-1}, \theta \circ \pi)$ is also a $C^r_b$ solution pair.
\end{remark}

Combining this remark with Theorem~\ref{thm:maps}\ref{thm:maps_2}, we obtain the following corollary.

\begin{corollary} 
Under the assumptions of Theorem~\ref{thm:maps},
    for any (local) $C^r_b$ diffeomorphism $\vartheta$ on $X_0$, there is a (locally) unique $C^r_b$ solution pair $(g,\pi)$ of \eqref{eq:map_foliation_diag} such that $\pi \vert_{X_0} = \vartheta$. 
\end{corollary}

\begin{remark}
    It should be noted that, by the uniqueness claim, the leaf passing through the origin, $\pi^{-1}(0)$, is the invariant manifold constructed in \cite{de1997invariant}.
    In comparison, due to the invariance relation \eqref{eq:invariance_maps} taking a more workable form,
    $C^r$ smoothness becomes much easier to achieve in the present setting (cf.\ Lemma~\ref{lemma:Tfrechet}), but this comes at the expense of a slightly more restrictive spectral assumption in \ref{hyp4:spec1} (compare with hypothesis (vi) of Theorem 2.1, \cite{de1997invariant}).
\end{remark}

\begin{remark}
    The regularity $r$ could be permitted to include non-integer values (which stand for the $(r-\lfloor r \rfloor)$-Hölder continuity of the $\lfloor r \rfloor$-th derivative).
    This change corresponds to only minor adjustments throughout the proof (cf.\ \cite{kvalheim2021existence}), in particular,
    the replacement of results of \cite{Irwin72}, at the points where we have used them, by those of \cite{de1998regularity} (e.g., in Lemma~\ref{lemma:Tfrechet}).
\end{remark}

A \textit{semigroup} (of operators on a Banach space $X$) is a map $T: \mathbb{R}^{\geq 0} \to \mathcal{L}(X)$ satisfying the semiflow property \eqref{eq:sf}.
A semigroup is \textit{strongly continuous} if it moreover satisfies
\begin{displaymath}
    \lim_{t \searrow 0}\vert T_t x - x \vert = 0, \qquad \text{for all } x \in X.
\end{displaymath}
The \textit{infinitesimal generator} of $T$ is the linear operator $G : \mathrm{dom}(G) \to X$ given by
\begin{displaymath}
    Gx = \lim_{t \searrow 0} \frac1t (T_t x - x)
\end{displaymath}
on its domain 
\begin{displaymath}
    \mathrm{dom}(G) = \left\{x \in X \; \Big| \; \text{the limit } \lim_{t \searrow 0} \frac1t (T_t x - x) \text{ exists} \right\}.
\end{displaymath}
If $T:\mathbb{R} \to \mathcal{L}(X)$ satisfies \eqref{eq:sf} for all $t,s \in \mathbb{R}$, we call it a \textit{group} (of operators on $X$).

A strongly continuous semigroup $T_t$ is called analytic (roughly speaking) if the function $t \mapsto T_t x$ extends to a holomorphic function on a cone about the positive real axis (in $\mathbb{C}$) for all $x \in X$.
An operator $G$ is called sectorial 
if the complement of its spectrum contains a set of the form 
\begin{equation}
    \big\{ \, \lambda \in \mathbb{C} \; \big\vert \;
    | \mathrm{arg} (\lambda-\omega) | < \theta \, \big\}
    \label{eq:sect_domain}
\end{equation}
for some $\omega \in \mathbb{R}$ and $\theta \in (\pi/2,\pi)$, and moreover, the resolvent of $G$ is bounded by $C/ |\lambda-\omega|$ over \eqref{eq:sect_domain} for some uniform constant $C$.
A classical result in semigroup theory says that an operator $G$ with closed graph and dense domain defines an analytic semigroup if and only if it is sectorial.
Throughout this work we will assume that all semigroups are at least strongly continuous, and hence that generators have closed graphs and dense domains.

The following set of assumptions are for the semiflow version of Theorem~\ref{thm:maps}.

\begin{enumerate}[label =(A.\arabic*)] 

\item \label{hypA1} 
Let $(X, | \cdot | )$ denote a Banach space.
Let $\varphi$ be a local semiflow with domain $\mathcal{D}^\varphi \subset \mathbb{R}^{\geq 0} \times X$ and a Lyapunov stable (with respect to $| \cdot |$) fixed point at $0$.
Suppose that, for each fixed $t \geq 0$, the map
$x \mapsto \varphi_t(x)$ is of class $C^r_b$ on $\mathcal{D}_t^\varphi \cap O$, for some fixed open neighborhood $O$ of $0$, and $r \in \mathbb{N} \cup \{\omega,\infty \}$.

\item \label{hypA2} 
Let $X$ and $\varphi$ be as in \ref{hypA1}, and assume in addition that $X$ admits a splitting as in \ref{hyp1:X}.
Suppose that $t \mapsto D \varphi_t(0)$ is a strongly continuous semigroup on $X$ with generator $G$ such that $X_0 \subset \mathrm{dom}(G)$ and $D \varphi_t(0) (X_1) \subset X_1$ for all $t \geq 0$.
Assume there exists $\tau > 0$ such that $\sigma(D\varphi_\tau(0)) \subset B_1^{\mathbb{C}}$ and for which \ref{hyp3:ell} and \ref{hyp4:spec1} hold (with $A_0 = \widetilde{P}_0 D\varphi_\tau(0) \imath_0 $ and $A_1 = \widetilde{P}_1 D\varphi_\tau(0) \imath_1 $). 
Suppose moreover
that $(t,x_0) \mapsto D^i (\varphi_t \circ \imath_0) (x_0)$ is jointly continuous on the line $\mathbb{R}^{\geq 0} \times \{ 0 \}$ for all $0 \leq i \leq \ell + 1$ (with $\ell$ as in \ref{hyp3:ell} above).

\item \label{hypA3} 
The spectrum of $D \varphi_\tau (0)$ satisfies \ref{hyp5:spec2} for $m= 2$ (with $A_0 = \widetilde{P}_0 D\varphi_\tau(0) \imath_0 $).

\end{enumerate}

\begin{theorem} \label{thm:sf}
    Suppose hypotheses \ref{hypA1} and \ref{hypA2}.
    Then, the following assertions hold.
    \begin{enumerate}[label =({\roman*})]
    \item \label{thm:sf_1} There exist open neighborhoods of $0$, $V \subset U \subset X$, a $C^r_b$ map $\pi : U \to X_0$ with 
    \begin{subequations}\label{eq:sf_foliation_diag}
    \begin{equation}
        \pi(0) = 0, \qquad D\pi(0) \vert_{X_0} \in \mathrm{Aut}(X_0),
    \end{equation}
    and a $C^r_b$ diffeomorphism $\vartheta_t$
    such that
    \begin{equation}
    \begin{tikzcd} 
    V \arrow[r, "\varphi_t"] \arrow[d, "\pi"']
    & U \arrow[d, "\pi"] \\
    \pi (V)  \arrow[r, "\vartheta_t"']
    &  \pi(U)
    \end{tikzcd}
    \end{equation}\end{subequations}
    commutes for all $t \geq 0$, with $\vartheta$ satisfying the semiflow property \eqref{eq:sf}.
    If $\varphi \vert_{\mathcal{D}^\varphi \cap (\mathbb{R}^{\geq 0} \times X_0)}$ is jointly $C^k$, $0 \leq k \leq r$, then $\vartheta$ extends to define a jointly $C^k$ flow on an open domain $\mathcal{D}^\vartheta \subset \mathbb{R} \times X_0$.
    One may alternatively assume that 
    $X_0$ is finite dimensional and that $t \mapsto \varphi_t(x_0)$ is continuous for all $x_0 \in X_0$ (i.e., mere separate continuity is assumed over $\mathbb{R}^{\geq 0} \times X_0$) to conclude that $\vartheta$ extends to a jointly $C^r$ flow, if $r \in \mathbb{N} \cup \{ \infty \}$. 
    If \ref{hypA3} holds, then one may choose $\vartheta : \mathbb{R} \to \mathrm{Aut}(X_0)$ to be a group (of operators on $X_0$) with the same smoothness.
    \item \label{thm:sf_2} 
    The solution constructed in \ref{thm:sf_1} is unique among $C^{\ell+1}_b$ solutions in the following sense. 
    Given two pairs of $C^k_b$ (in space), $\ell + 1 \leq k \leq r$ ($k = \omega$ if $r = \omega$), solutions to \eqref{eq:sf_foliation_diag}, $(\vartheta,\pi)$ and $(\tilde{\vartheta},\tilde{\pi})$, there exists a 
    unique
    local $C^{k}_b$ diffeomorphism $\theta$ on $X_0$ such that
    \begin{equation}
        \tilde{\vartheta_t} = \theta \circ \vartheta_t \circ \theta^{-1}, \quad t \geq 0\footnote{If $\vartheta$ and $\tilde{\vartheta}$ are flows, $t \in \mathbb{R}$.}, \qquad \tilde{\pi} = \theta \circ \pi.
        \label{eq:uniqueness_sf}
    \end{equation} 

    If both $\vartheta_t$ and $\tilde{\vartheta_t}$ are linear, then so is $\theta$.
    \end{enumerate}
\end{theorem}

\begin{proof}
    The proof is deferred to Section~\ref{sect:proof}.
\end{proof}

It is often preferable to formulate spectral hypotheses (as in \ref{hypA2} and \ref{hypA3}) in terms of the generator $G$ of the linearized semiflow $t \mapsto D \varphi_t(0)$ instead of its time-$\tau$ map $D \varphi_\tau(0)$. 
For this, we need the following result.

\begin{theorem}[Spectral mapping for semigroups, Theorems 3.7 and 3.10 of \cite{engel2000one}] \label{thm:SM}
    For a generator $G: \mathrm{dom}(G) \to X$ of a strongly continuous semigroup $\{T_t \}_{t \geq 0}$ on a Banach space $X$, we have the identities
    \begin{subequations} \label{eq:SM_point_residual}
    \begin{align}
        \mathrm{P} \sigma ( T_t ) \setminus \{ 0\} &= e^{t \mathrm{P} \sigma (G)}, \\ 
        \mathrm{R} \sigma ( T_t ) \setminus \{ 0\} &= e^{t \mathrm{R} \sigma (G)}, 
    \end{align}\end{subequations}
    for all $t \geq 0$.
    If in addition $\{T_t \}_{t \geq 0}$ is eventually norm-continuous, then
    \begin{equation}
        \sigma ( T_t ) \setminus \{ 0\} = e^{t \sigma (G)}, \qquad t \geq 0, 
        \label{eq:SM_norm_cont}
    \end{equation}
    holds.
\end{theorem}

Analytic semigroups form one class of examples that are eventually norm-continuous,
and they crop up naturally in the context of parabolic semilinear PDEs on bounded domains \cite{henry1981geometric,lunardi1995analytic}, which we anticipate to be the most common area of application for the results herein. 
They moreover have the advantageous property that they may be written as an exponential function of a sectorial operator, making it easy to enforce the required form of $D \varphi_\tau (0)$ and the $X_1$ invariance via the generator.
We shall therefore aim to formulate assumptions for sectorial generators.

\begin{remark}[Remark~3, \cite{kvalheim2021existence}] \label{remark:spectral}
    If the generator $G$ of $D \varphi_t(0)$ is sectorial, and moreover, $X_0$ is finite dimensional, then we may impose the spectral conditions \ref{hyp4:spec1} and \ref{hyp5:spec2} directly on $G$.
    First, note that if \ref{hypA2} is satisfied, then $G$ must be of the form
\begin{displaymath} 
    G = \begin{pmatrix}
        G_0 & 0 \\
        B & G_1
    \end{pmatrix}
\end{displaymath}
with respect to the splitting $X = X_0 \oplus X_1$.
    Now, by \eqref{eq:SM_norm_cont}, if the nonresonance condition \ref{hyp4:spec1} (or \ref{hyp5:spec2}) is to fail for all $\tau > 0$ (recall $\tau > 0$ is arbitrary in \ref{hypA2}), then for any sequence $\tau_k \to 0$ there exist index sets $I_k$ and $J_k$ labeling elements of $\sigma(G_0)$ and $\sigma(G_1)$ respectively, $m_k \in \mathbb{Z}$
    and $\nu_k \in \sigma(G_0)$ such that
    \begin{equation}
        \sum_{i \in I_k} \lambda_i + \sum_{i \in J_k} \mu_i = \nu_k + \frac{2 \pi \mathrm{i} m_k}{\tau_k}.
        \label{eq:spectral_log}
    \end{equation} 
    Since $\sigma(G_0)$ and $\ell$ from \ref{hyp3:ell} are finite, we may choose a subsequence such that $I_k = I$, $\nu_k = \nu$ and $\# J_k < \infty$ are all fixed.
    If $m_k = 0$ for any $k$ we are done, so assume the contrary.
    Then \eqref{eq:spectral_log} and the spectral stability assumption implies that there exists a sequence of elements of $\sigma(G_1)$ with bounded real part and unbounded imaginary part, contradicting the sectoriality of $G$. 
\end{remark}

We have thus obtained the following set of alternative (sufficient but not necessary) assumptions in place of \ref{hypA2}-\ref{hypA3} for Theorem~\ref{thm:sf}.

\begin{enumerate}[label =(\~{A}.\arabic*)] 
\setcounter{enumi}{1}

\item \label{hypA2b} 
Let $X$ and $\varphi$ be as in \ref{hypA1}, and assume in addition that $X$ admits a splitting as in \ref{hyp1:X}.
Suppose that $(t,x_0) \mapsto D^i (\varphi_t \circ \imath_0) (x_0)$ is jointly continuous on the line $\mathbb{R}^{\geq 0} \times \{ 0 \}$ for all $0 \leq i \leq \ell+1$ (with $\ell$ as in \ref{hypA3b});
that $t \mapsto D \varphi_t(0)$ is an analytic semigroup on $X$ with generator $G$ taking the form
\begin{equation} \label{eq:G_sf}
    G = \begin{pmatrix}
        G_0 & 0 \\
        B & G_1
    \end{pmatrix}
\end{equation}
with respect to the splitting $X = X_0 \oplus X_1$ of \ref{hyp1:X}, such that $X_0 \subset \mathrm{dom}(G)$.
Assume moreover that $G$ has compact resolvent and $\sigma(G) \subset \{ z \in \mathbb{C} \, | \, \mathrm{Re} \, z < 0 \}$.

\item \label{hypA3b} Let $1 \leq \ell <r$ be an integer such that
\begin{displaymath}
    (\ell +1) \sup \{ \mathrm{Re} \, \lambda \, | \, \lambda \in \sigma(G) \} < \inf \{ \mathrm{Re} \, \lambda \, | \, \lambda \in \sigma(G_0) \}, 
\end{displaymath}
with $G$ as in \eqref{eq:G_sf}.

\item \label{hypA4b} The spectrum of $G$ from \eqref{eq:G_sf} satisfies 
\begin{displaymath}
   \left( (n-j) \sigma(G_0) +
    j\sigma(G_1) \right) \cap \sigma(G_0) = \emptyset
\end{displaymath}
for all pairs $(j,n)$ of integers such that $2 \leq n \leq \ell$ and $1 \leq j \leq n$, with $\ell$ as in \ref{hypA3b}.

\item \label{hypA5b} The spectrum of $G$ from \eqref{eq:G_sf} satisfies 
\begin{displaymath}
    n\sigma(G_0)  \cap \sigma(G_0) = \emptyset
\end{displaymath}
for all $2 \leq n \leq \ell$, with $\ell$ as in \ref{hypA3b}.

\end{enumerate}

The hypotheses of Theorem~\ref{thm:sf} may be replaced by \ref{hypA1}, \ref{hypA2b}-\ref{hypA4b}, along with the optional assumption \ref{hypA5b} for linearization.

As already indicated in Section~\ref{sect:intro_tech}, the choice of the target space $X_0$ for the submersions was arbitrary in Theorems \ref{thm:maps} and \ref{thm:sf}.
We shall make this a little more precise, as it will be needed in the following section.

\begin{remark} \label{remark:mfd_exist_uniq} 
    We could have equivalently formulated 
    Theorems \ref{thm:maps} and \ref{thm:sf} for any
    $C^r$ smooth manifold $M_0$ 
    in place of $X_0$ (the assertion that $g, \vartheta_t$ can be chosen to be polynomial only holds if $M_0$ is a linear manifold). 
    In this case, the assumption prescribing the form of $G$, \eqref{eq:G_sf} (and similarly for \eqref{eq:A}), is replaced by
    \begin{equation}
        KG = G_0 K,
        \label{eq:G_alt_form}
    \end{equation}
    for $K : X \to T_pM_0$ a projection, for some $p \in M_0$. 
    A projection in this setting amounts to the existence of a splitting $X_0 \oplus X_1 = X$ such that $K$ can be factored as $X_0 \oplus X_1 \to X_0 \to  T_pM_0$, where the first map is the canonical projection and the second one is an isomorphism.
    Consequently, we may view $M_0$ locally about $p$ as being modeled on $X_0$; fix a chart $\sigma$ mapping $p$ to $0$. 
    Since the spectral assumptions placed on $G_0$ are left invariant by the isomorphism, Theorem~\ref{thm:sf} is applicable with $X_0$ and $X_1$ as above, so long as $X_0$ satisfies the additional requirements posed by \ref{hypA2b}.
    Then,
    if $(\vartheta,\pi)$ is a solution to 
    \eqref{eq:sf_foliation_diag}, we may assume $\mathrm{im}(\sigma) \supset \pi(U)$ (up to shrinking $V$ and hence $U$, with notations as in \eqref{eq:sf_foliation_diag}), 
    then existence follows once again from the observation that $\sigma^{-1} \circ \pi$ is a submersion and 
    \begin{displaymath}
        (\sigma^{-1} \circ \vartheta_t \circ \sigma, 
        \sigma^{-1} \circ \pi)
    \end{displaymath}
    also satisfies the invariance relation \eqref{eq:intro_invariance} on $V$. 
    The uniqueness claim continues to hold (locally)  in  $C^{k}_b$, $\ell +1 \leq k \leq r$, with respect to different choices of the  target manifold $M_0$,
    for if $(\pi,\vartheta_t)$ and $(\widetilde{\pi},\widetilde{\vartheta}_t)$ are two $C^{k}_b$ solution pairs with target spaces $M_0$ and $\widetilde{M}_0$ respectively, with two $C^{k}_b$ charts $\sigma: M_0 \supset O \to X_0$ and $\widetilde{\sigma}: \widetilde{M}_0 \supset \widetilde{O} \to X_0$ mapping the respective fixed points to $0$, 
    then \eqref{eq:uniqueness_sf} gives a locally unique $C^k_b$ diffeomorphism $\theta$ around $0$ on $X_0$ such that
    \begin{subequations}
    \begin{align} 
         \widetilde{\sigma} \circ \widetilde{\vartheta}_t \circ \widetilde{\sigma}^{-1} &= \theta \circ \sigma \circ \vartheta_t \circ \sigma^{-1} \circ \theta^{-1}, \\ 
         \widetilde{\sigma} \circ \widetilde{\pi} &= \theta \circ \sigma \circ \pi
        \label{eq:uniqueness_mfd}
    \end{align}\end{subequations} 
    holds locally around $0$ in $X_0$ and $X$ respectively.
    The transition diffeomorphism $\widetilde{\sigma}^{-1} \circ \theta \circ \sigma$ is (locally) uniquely determined by \eqref{eq:uniqueness_mfd}, irrespective of the choice of charts. 
\end{remark}

In other words, what this means is that the uniqueness claims in Theorems \ref{thm:maps} and \ref{thm:sf} actually imply 
the uniqueness of the foliation defined by $\pi$ within the subclass of
$C^{\ell+1}_b$ invariant foliations tangent to $X_1$ at $0$.

\section{Koopman eigenfunctions}
\label{sect:koopman}

The preceding theory of invariant foliations has significant consequences 
extending to the realm of Koopman theory \cite{koopman1931hamiltonian}, a brief introduction to which follows.
The definitions recalled here are in line with \cite{mezic2020spectrum} 
and \cite{kvalheim2021existence,mohr2016koopman}, subject to minor modifications to conform to the setting of semiflows.

\begin{definition}[Koopman eigenfunctions] \label{def:koopman}
    Let \{$\varphi_t\}_{t \geq 0}$ denote a (local) semiflow on a Banach space $X$.
    Let $O \subset X$ be an open subset, and let
   \begin{equation}
        \mathcal{O} = \Big\{ (t,x) \in \mathbb{R}^{\geq 0} \times O \; \big\vert  \; \varphi_s(x) \in O \text{ for all } s \in [0,t] \Big\}. 
        \label{eq:Koopmandomain}
    \end{equation}
    Let $\mathbb{F}$ denote a field, taken to be  $\mathbb{C}$ at all times except for when both $X$ and $\lambda$ below are real, in which case $\mathbb{F} = \mathbb{R}$.
    A function $\psi: O \to \mathbb{F}$ 
    that is not identically zero is called an \textit{open Koopman eigenfunction with eigenvalue $\lambda \in \mathbb{C}$}
    if
    \begin{equation}
        \psi \circ \varphi_t (x) = e^{\lambda t} \psi(x)
        \label{eq:Koopman_def}
    \end{equation}
    holds for all pairs $(t,x) \in \mathcal{O}$.
    If $\mathcal{O} = \mathbb{R}^{\geq 0} \times O$, 
    $\psi$ is simply called a \textit{Koopman eigenfunction} (with eigenvalue $\lambda$).\footnote{These are sometimes referred to as subdomain eigenfunctions if $O$ is a proper subset of $X$ \cite{mezic2020spectrum}; we will not bother to make the distinction here.}
\end{definition}

\begin{remark} \label{remark:koopman_domain}
    Note that the domain $\mathcal{O}$ given in \eqref{eq:Koopmandomain} of Definition~\ref{def:koopman} is equivalent to the one specified in Definition~\ref{def:foliation_invariance} for locally invariant foliations.
    It is an open subset of $\mathbb{R}^{\geq 0} \times O \cap \mathcal{D}^\varphi$ whenever $\varphi : \mathcal{D}^\varphi \to X$ is continuous. 
    To see this, take $(t,x) \in \mathcal{O}$, and
    observe that $\varphi^{-1}(O)$ is an open subset of $\mathcal{D}^\varphi$ containing the compact set $[0,t] \times \{ x \}$. 
    By compactness, we may construct an open neighborhood of $[0,t] \times \{ x \}$ that is of the form $[0,t+\varepsilon) \times B^X_\varepsilon(x) \subset \varphi^{-1}(O)$, thereby providing an open neighborhood of $(t,x)$ in $\mathcal{O}$. 
\end{remark}

Koopman eigenfunctions are usually sought in more restrictive subspaces of functions $O \to \mathbb{F}$, to make them appropriate for their envisioned purpose. 
One normally requires $\psi$ to be at least continuous, but often times a higher degree of regularity is desirable. 
When drawing relations to the preceding theory of invariant foliations about attracting equilibria, we consider eigenfunctions that are at least $C^1$.
For the case of unstable fixed points, this will have to be degraded to $C^0$, as this seems to be the best we can presently achieve.

\begin{definition}[Principal Koopman eigenfunctions; Definition~4, \cite{kvalheim2021existence}]
Let \{$\varphi_t\}_{t \geq 0}$ denote a (local) semiflow on a Banach space $X$ with a fixed point at $0$.
    Let $\psi:O \to \mathbb{F}$ denote a $C^1$ ($C^\omega$ if $X$ is complex) (open) Koopman eigenfunction as in Definition~\ref{def:koopman}, such that $0 \in O$.
    We say $\psi$ is an \textit{(open) principal Koopman eigenfunction at $0$} if $\psi(0) = 0$ and $D\psi (0) \neq 0$.
\end{definition} 

\begin{remark} \label{remark:koopman_ext}
    If $0$ is an asymptotically stable fixed point, Koopman eigenfunctions (and linearizing semiconjugacies) are extendable in a unique way to its basin of attraction, $\mathcal{B}(0)$.
    The proof of this claim is given in
    \cite{lan2013linearization,kvalheim2021existence} for the finite-dimensional setting; it carries over verbatim to the present one.
\end{remark}

The following lemma is an immediate consequence of the definitions.

\begin{lemma} \label{lemma:pkoop_in_spect}
    Suppose $\varphi$ is a $C^1$-in-space ($C^\omega$ in the complex case) local semiflow on a real or complex Banach space $X$ with a fixed point at $0$, such that $t \mapsto D \varphi_t(0)$ is analytic; denote its generator by $G$.
    Given an (open) principal Koopman eigenfunction $\psi$ at $0$ with eigenvalue $\lambda$,
    we have $\lambda \in \mathrm{P}\sigma (G) \cup \mathrm{R}\sigma (G)$. 
\end{lemma}

\begin{proof} 
    Since, for all $t \geq 0$, $\mathcal{O}_t\footnote{Here, $\mathcal{O}_t = \{x \in O \, | \, (t,x) \in \mathcal{O} \}$.} \subset \mathcal{D}^\varphi_t$ is an open neighborhood of $0$ in $X$ on which \eqref{eq:Koopman_def} holds, we may differentiate \eqref{eq:Koopman_def} at $0$ to obtain
    \begin{equation} 
        D \psi (0) D \varphi_t (0) = e^{\lambda t} D \psi (0), \qquad t \geq 0.
        \label{eq:koopmanlinear}
    \end{equation}
    If $X$ is complex and $\psi \in C^\omega$, we have that $D\psi(0) \in X^*$. 
    It follows that $e^{\lambda t} \in \mathrm{P}\sigma ((D \varphi_t (0))^*)$\footnote{Here, $(\cdot)^*$ refers to the dual operator acting on the dual space $X^*$.}, which implies 
    $e^{\lambda t} \in \mathrm{P}\sigma (D \varphi_t (0)) \cup \mathrm{R}\sigma (D \varphi_t (0))$, $t \geq 0$ (see Lemma 6.2.6, \cite{buhler2018functional}). 
    If $X$ is instead real, we may consider its complexification $X^c$ along with the complexified operators $(D \varphi_t(0))^c$ and $(D \psi(0) )^c \in (X^c)^*$  to conclude the same.

    Now, by \eqref{eq:SM_point_residual} of Theorem~\ref{thm:SM} and the same argument as in Remark~\ref{remark:spectral}, we may conclude that $\lambda \in \mathrm{P} \sigma(G) \cup \mathrm{R} \sigma(G)$.
\end{proof}

Within the main scope of applications considered herein, i.e., parabolic PDEs on bounded domains, $G$ is an elliptic operator that has compact resolvent (and is sectorial, \cite{agmon1959estimates,lunardi1995analytic}).
Then, $\sigma(G) = \mathrm{P}\sigma(G)$, and $\lambda$ is a genuine eigenvalue of $G$.
Otherwise, we are met with the slightly intriguing possibility of the contrary, that
a Koopman eigenvalue could be in the residual spectrum of $G$.

Denote the set of $C^k$ (open) principal Koopman eigenfunctions at $0$ by $\mathcal{K}^k(0)$ (resp.\ $\mathcal{K}^k_o(0)$). 
Elements of $\mathcal{K}^k(0)$ are of the form $(\psi,\lambda)$, standing for eigenfunction-eigenvalue pairs.
It will be convenient later on to equate eigenfunctions differing by a choice of observable, i.e., to consider the equivalence relation $\sim$ on $\mathcal{K}^k(0)$ (and $\mathcal{K}^k_o(0)$) given, for $\psi_i : O_i \to \mathbb{F}$ and $\lambda_i \in \mathbb{C}$, $i =1,2$, by
\begin{align*}
    (\psi_1,\lambda_1) \sim (\psi_2,\lambda_2) \qquad   \iff \qquad &\lambda_1 = \lambda_2 \text{ and }  \exists c \in \mathbb{F} \text{ and an open set } O \subset O_1 \cap O_2  \\   &\text{containing $0$ such that }  \psi_1 = c \psi_2 \text{ on } O.
\end{align*} 
Denote by $\mathscr{K}^k(0) = \mathcal{K}^k(0) / \sim$
(and similarly, $\mathscr{K}^k_o(0) = \mathcal{K}^k_o(0) / \sim$ for open eigenfunctions).
Denote further by
\begin{equation}
    \varpi : a + \mathrm{i}b \mapsto (a,b)
    \label{eq:iso_C_R2}
\end{equation}
the isomorphism $\mathbb{C} = \mathbb{R} \oplus \mathrm{i} \mathbb{R} \cong \mathbb{R}^2$.

The following lemma shows how principal Koopman eigenfunctions define foliations.

\begin{lemma} \label{lemma:koopman_foli_1_1_new}
    Given a semiflow $\varphi$ on a Banach space $X$ with a fixed point at $0$, consider its set of principal Koopman eigenfunctions $\mathscr{K}^k(0)$ introduced above, $k \in \mathbb{N} \cup \{\infty,\omega\}$ ($k = \omega$ if $X$ is complex).
    For each $[(\psi,\lambda)] \in \mathscr{K}^k(0)$ (resp.\ $[(\psi,\lambda)] \in \mathscr{K}^k_o(0)$) there exists a unique $C^k$ (resp.\ locally) invariant foliation about $0$ representing it. 
    Moreover, if $[(\psi_1,\lambda_1)]$ and $[(\psi_2,\lambda_2)]$ are locally represented by the same foliation, then $\lambda_1 = \lambda_2$ or $\lambda_1 = \overline{\lambda}_2$; 
    if $\mathrm{Re} \, \lambda_1 <  0$ and the $\mathscr{K}^k(0)$ case is assumed, then 
    $[(\psi_1,\lambda_1)] = [(\psi_2,\lambda_2)]$ or $[(\psi_1,\lambda_1)] = [(\overline{\psi}_2,\overline{\lambda}_2)]$. 
    If $X$ is complex and $ k = \omega$, then it must be the $[(\psi_1,\lambda_1)] = [(\psi_2,\lambda_2)]$ case. 
\end{lemma}

\begin{proof}
    The first assertion is obvious: if $[(\psi,\lambda)] \in \mathscr{K}^k_o(0)$, then $D\psi (0)$ is bounded and linear, hence $\ker ( D\psi(0))$ is a closed linear subspace of finite codimension, thereby complemented in $X$.
    Thus, $\psi$ is a submersion at $0$ that satisfies the invariance relation \eqref{eq:Koopman_def} on the domain $\mathcal{O}$ (as defined in \eqref{eq:Koopmandomain}, see also Remark~\ref{remark:koopman_domain}), 
    and hence defines a $C^k$ (locally) invariant foliation about $0$.  
    Clearly, this foliation is uniquely determined, independent of the representative chosen in $\mathscr{K}^k_o(0)$.

    For the converse, we treat the case $\mathbb{F} = \mathbb{C}$; the $\mathbb{F} = \mathbb{R}$ case is simpler and follows analogously.
    If $[(\psi_1,\lambda_1)]$ and $[(\psi_2,\lambda_2)]$ are locally represented by the same $C^k$ foliation, then there exists a $C^k$ diffeomorphism $h$ of $\mathbb{C}$ (locally about $0$)--which is to be interpreted under the identification $\mathbb{C} = \mathbb{R} \oplus \mathrm{i} \mathbb{R} \cong \mathbb{R}^2$ for general $k$--such that $\psi_2 = h \circ \psi_1$
    holds on some open neighborhood of the origin, $O$, which we may assume to be the domain of both $\psi_1$ and $\psi_2$; denote their corresponding 'Koopman' domains \eqref{eq:Koopmandomain} (i.e., on which \eqref{eq:Koopman_def} holds) by $\mathcal{O}^1$ and $\mathcal{O}^2$.
    Then,
    \begin{equation}
        h \circ e^{ t \lambda_1 } =e^{ t\lambda_2 } h
        \label{eq:h_conj}
    \end{equation}
    holds on $\psi_1 (\mathcal{O}^1_t \cap \mathcal{O}^2_t)$, which is an open neighborhood of $0$ in $\mathbb{C}$ for all $t \geq 0$ fixed.  
    Hence, we may differentiate \eqref{eq:h_conj} (in space) at $0$ to obtain
    \begin{equation}
        D(\varpi h \circ \varpi^{-1})(0) 
        \begin{pmatrix}
            e^{t \, \mathrm{Re} \, \lambda_1} & -e^{t \, \mathrm{Im} \, \lambda_1} \\
            e^{t \, \mathrm{Im} \, \lambda_1} & e^{t  \, \mathrm{Re} \, \lambda_1}
        \end{pmatrix}
        = 
        \begin{pmatrix}
            e^{t \, \mathrm{Re} \, \lambda_2} & -e^{t \, \mathrm{Im} \, \lambda_2} \\
            e^{t \, \mathrm{Im} \, \lambda_2} & e^{t \, \mathrm{Re} \, \lambda_2}
        \end{pmatrix}
        D(\varpi h \circ \varpi^{-1})(0) 
        \label{eq:h_conj2}
    \end{equation}
    for all $t \geq 0$, 
    where we have used the identification \eqref{eq:iso_C_R2} to make sense of the derivative of $h$.
    Proceeding to differentiate \eqref{eq:h_conj2} in time at $t = 0$, we have
    \begin{displaymath}
        D(\varpi h \circ \varpi^{-1})(0) 
        \begin{pmatrix}
             \mathrm{Re} \, \lambda_1 & - \mathrm{Im} \, \lambda_1 \\
             \mathrm{Im} \, \lambda_1 &  \mathrm{Re} \, \lambda_1
        \end{pmatrix}
        = 
        \begin{pmatrix}
             \mathrm{Re} \, \lambda_2 & - \mathrm{Im} \, \lambda_2 \\
             \mathrm{Im} \, \lambda_2 &  \mathrm{Re} \, \lambda_2
        \end{pmatrix}
        D(\varpi h \circ \varpi^{-1})(0),
    \end{displaymath}
    which, since $D(\varpi h \circ \varpi^{-1})(0) \in \mathrm{GL}(\mathbb{R}^2)$, implies that the spectra of the two matrices composed of $\lambda_i$ ($i = 1,2$) must agree.
    We have thus shown that either $\lambda_1 = \lambda_2$ or $\lambda_1 = \overline{\lambda}_2$.
    It follows that $D(\varpi h \circ \varpi^{-1})(0) = \mathrm{id}_{\mathbb{R}^2}$ and 
    $
    D(\varpi h \circ \varpi^{-1})(0) = 
    \begin{pmatrix}
        1 & 0 \\ 0 & -1
    \end{pmatrix}
    $ 
    are the  solutions to \eqref{eq:h_conj2}, corresponding to $\lambda_1 = \lambda_2$ and $\lambda_1 = \overline{\lambda}_2$ respectively, each defined up to multiplication by elements of the form $\begin{pmatrix}
       a & -b \\ b & a 
    \end{pmatrix}$  (which corresponds to complex multiplication once transformed back via $\varpi$).
    
    Fix any solution $D(\varpi h \circ \varpi^{-1})(0) = L \in \mathrm{GL}(\mathbb{R}^2)$ of \eqref{eq:h_conj2}.
    If $\mathrm{Re} \, \lambda_1 < 0$ and $\mathcal{O}^1$ is of the form $\mathbb{R}^{\geq 0} \times O^1$ (i.e., the $\mathscr{K}^k(0)$ case), \eqref{eq:h_conj} above holds on a uniform neighborhood over all $t \geq 0$, which, along with the uniqueness claim 
    \ref{thm:maps_2} of Theorem~\ref{thm:maps} (or perhaps more directly, Lemma~\ref{lemma:uniq1}) applied to \eqref{eq:h_conj}, implies that the solution $h = Dh(0)$ 
    is (locally) the unique one satisfying \eqref{eq:h_conj} with $D(\varpi h \circ \varpi^{-1})(0) = L$.     
    Upon descending to equivalence classes, this yields that the only two possible scenarios are $[(\psi_1,\lambda_1)] = [(\psi_2,\lambda_2)]$ or $[(\psi_1,\lambda_1)] = [(\overline{\psi}_2,\overline{\lambda}_2)]$.
    If $X$ is complex and $ k = \omega$, then $h$ in \eqref{eq:h_conj} is complex analytic and hence $Dh(0) \in \mathbb{C}$, implying that the prior must occur. 
\end{proof}

We remark that if $X$ is real, then we necessarily have that complex Koopman eigenvalues and eigenfunctions appear in conjugate pairs; Lemma~\ref{lemma:koopman_foli_1_1_new} implies that these are the only pairs with coinciding foliations.

\begin{corollary} \label{cor:koop_foli_1_1}
    Let $\varphi$ be a $C^\omega$-in-space, $C^0$-in-time semiflow on a complex Banach space $X$, with a spectrally stable fixed point at $0$.
    Denote by $\mathscr{F}^\omega_1(0)$ the set of codimension-one $C^\omega$ invariant foliations about $0$ (with elements defined up to restriction of domain).
    Then, there is a bijection between the sets $\mathscr{F}^\omega_1(0)$ and $\mathscr{K}^\omega(0)$.
\end{corollary}

\begin{proof}
    Injectivity and well-definedness of the map $\mathscr{K}^\omega(0) \to \mathscr{F}^\omega_1(0)$ was shown in Lemma~\ref{lemma:koopman_foli_1_1_new}.
    We show surjectivity.
    
    Given a $C^\omega$ invariant codimension-one foliation about $0$, we may view it as a $C^\omega$ submersion $\pi : X \supset U \to X_0$, for $X_0 \cong \mathbb{C}$ a one-dimensional subspace complementing $\mathrm{ker} ( D \pi (0) )$, such that
    \begin{equation}
        \pi \circ \varphi_t = \vartheta_t \circ \pi
        \label{eq:invariance_equivproof}
    \end{equation}
    holds on $U$ for all $t \geq 0$, for some local semiflow $\vartheta$ on $X_0$
    (see the discussion in Section~\ref{sect:intro_tech}), which
    inherits the smoothness properties of $\varphi$. 
    
    By separate continuity in space and time, it follows that $\vartheta$ is jointly continuous via Theorem~(8A.3) of \cite{marsden1976hopf} (at $t = 0$ this relies on the finite dimensionality of $X_0$; see Remark~(8A.5;4) therein).
    Joint $C^\infty$-smoothness of $\vartheta$ then follows from its spatial smoothness by Theorems~(8A.6-7) of \cite{marsden1976hopf} (see also Remark~(8A.8;2) thereafter).
    In particular, since $t \mapsto D \vartheta_t (0)$ is a continuous semigroup on $X_0 \cong \mathbb{C}$, 
    there exists a unique $\lambda \in \mathbb{C}$ such that\footnote{See, e.g., Theorem~1.4 of \cite{engel2000one} for a direct proof of this claim.}
    \begin{displaymath}
        D \vartheta_t (0) = e^{\lambda t}, \qquad t \geq 0.
    \end{displaymath}

    As in the proof of Lemma~\ref{lemma:pkoop_in_spect}, we may differentiate \eqref{eq:invariance_equivproof} at $0$ to find that $e^{\lambda t} \in \mathrm{P} \sigma(D \varphi_t(0)) \cup \mathrm{R} \sigma(D \varphi_t(0))$, which, along with the assumptions, implies that $\lambda$ has a strictly negative real part.
    We may apply Theorem~\ref{thm:sf} on $X_0$ (with $X_1 = \{ 0 \}$) to obtain, locally around $0$, a $C^\omega$ linearizing conjugacy $h$, i.e.,
    \begin{displaymath}
        h \circ \vartheta_t \circ h^{-1} = e^{\lambda t}, \qquad t \geq 0.
    \end{displaymath}
    We have thus obtained a $C^\omega$ principal Koopman eigenfunction at $0$, given by $h \circ \pi$, that defines the original invariant foliation we have started off with. 
\end{proof}

\subsection{Stable case} \label{sect:stablecase}

We reproduce here Proposition~6 of \cite{kvalheim2021existence} for the case of semiflows on Banach spaces.
Rather than maintaining full generality in doing so, we instead tailor the results to the subcase most relevant in applications, the case of parabolic equations on bounded domains, so that they can be stated in a better suited form for ease of use.
If the application of interest lies outside this subcase,
we suggest using Theorem~\ref{thm:sf} directly,
which contains all the information needed to produce the results below.

\begin{corollary}  \label{cor:Koop_exist}
Let $\varphi$ be a local semiflow satisfying \ref{hypA1} on a real or complex Banach space $X$ (with $r = \omega$ if $X$ is complex).
Suppose $t \mapsto D \varphi_t (0)$ is an analytic semigroup on $X$ with generator $G$, such that $\sigma(G)$ is discrete and $\sigma(G) \subset \{ \mu \in \mathbb{C} \, |  \, \mathrm{Re} \, \mu < 0 \}$. 
Assume moreover the property that, for each $X_0 \subset \mathrm{dom}(G)$,  $(t,x_0) \mapsto D^i (\varphi_t \vert_{X_0})(x_0)$ is jointly continuous over the line $\mathbb{R}^{\geq 0} \times \{0 \}$ for all $0 \leq i \leq r$\footnote{If $r \in \{ \infty,\omega \}$, this line should be interpreted as 'for all $i \in \mathbb{N}_0$'.}.

Then, for any $\lambda \in \sigma(G)$ for which there exists $\ell \in \{1, \ldots, r-1 \}$ such that
\begin{subequations}\label{eq:nonres_koopman}
\begin{align}
    & (\ell + 1) \sup \{ \mathrm{Re} \, \mu \, | \, \mu \in \sigma(G) \} < \mathrm{Re} \, \lambda,  \label{eq:nonres_koopman_1} \\
    &\lambda \notin n \sigma(G), \qquad 2 \leq n \leq \ell,  \label{eq:nonres_koopman_2}
\end{align}\end{subequations}
hold, there exists an element $[(\psi,\lambda)]$ of $\mathscr{K}^r(0)$.
If in addition $\lambda$ is an algebraically simple eigenvalue, then $[(\psi,\lambda)]$ is uniquely determined. 
\end{corollary}

\begin{proof} 
    We treat the case when $X$ is a real Banach space, the complex analytic case is simpler and follows analogously.

    Consider the complexification $X^c =  X \oplus \mathrm{i}X$ of $X$ and $G^c$ of $G$. 
    Since $\sigma(G^c)$ is discrete by assumption, we have $\sigma(G^c) = \mathrm{P}\sigma(G^c) = \mathrm{P}\sigma((G^c)^*)$, and since $X$ was originally real, complex eigenvalues must appear in complex conjugate pairs.
    We may define a (finite rank) projection $P_\Omega : X^c \to X^c$ associated to the spectral subset $\Omega = \{ \lambda \} \cup \{ \overline{\lambda} \}$ via Dunford's integral, which, due to the inclusion of the conjugate pair, leaves the real subspace $X$ of $X^c$ invariant.

    We have (by, e.g., Proposition 2.3.4 of \cite{lunardi1995analytic})
    \begin{displaymath}
         P_\Omega e^{G^c t}  =  e^{\lambda t} \bigg( P_{\lambda} + \sum_{n=1}^\infty \frac{t^n}{n !} D_{\lambda}^n  \bigg) + (1-\delta_\lambda (\overline{\lambda})) e^{\overline{\lambda} t} \bigg( P_{\overline{\lambda}} + \sum_{n=1}^\infty \frac{t^n}{n !} D_{\overline{\lambda}}^n  \bigg), \qquad t \geq 0,
    \end{displaymath} 
    where $D_\mu = (G^c -  \mathrm{id}_X \mu) P_\mu$.
    Since $P\sigma((G^c \vert_{\mathrm{im}(P_\Omega)})^*) = \{ \lambda \} \cup \{ \overline{\lambda} \}$, we have that $\ker((D_\lambda\vert_{\mathrm{im}(P_\Omega)})^*)$ (and also $\ker((D_\lambda\vert_{\mathrm{im}(P_\lambda)})^*)$) is nonempty. 
    If $\lambda \neq \overline{\lambda}$, then $\mathrm{im}(P_\Omega) = \mathrm{im}(P_\lambda) \oplus \mathrm{im}(P_{\overline{\lambda}})$ and we may choose 
    a (not necessarily unique) $Q \in \ker((D_\lambda\vert_{\mathrm{im}(P_\Omega)})^*) \cap \mathrm{im}(P_{\overline{\lambda}})^\perp$; otherwise $P_\Omega = P_\lambda$, and we proceed with some $ Q \in \ker((D_\lambda\vert_{\mathrm{im}(P_\Omega)})^*)$.
    In either case, we have that
    \begin{displaymath}
        Q \widetilde{P}_\Omega e^{G^c t}  =  e^{\lambda t} Q \widetilde{P}_\Omega
    \end{displaymath}
    holds for $t \geq 0$.
    Let $K: = \varpi Q \widetilde{P}_\Omega \vert_X$, with $\varpi$ as in \eqref{eq:iso_C_R2}. 
    Then $K: X \to \mathbb{R}^2$ is a bounded linear operator satisfying
    \begin{displaymath}
        K e^{Gt}
        = 
        \begin{pmatrix}
            e^{t \, \mathrm{Re} \, \lambda} & -e^{t \, \mathrm{Im} \, \lambda} \\
            e^{t \, \mathrm{Im} \, \lambda} & e^{t \, \mathrm{Re} \, \lambda}
        \end{pmatrix}
        K
    \end{displaymath}
    for $t \geq 0$, which implies \eqref{eq:G_alt_form} holds on $\mathrm{dom}(G)$ with $G_0 = 
    \begin{pmatrix}
         \mathrm{Re} \, \lambda & - \mathrm{Im} \, \lambda \\
         \mathrm{Im} \, \lambda &  \mathrm{Re} \, \lambda
    \end{pmatrix}
    $.
    (If $\lambda \in \mathbb{R}$, then $Q$ can be chosen such that $\mathrm{im}( Q \widetilde{P}_\Omega \vert_X ) \subset \mathbb{R} \subset \mathbb{C}$; in this case $K : X \to \mathbb{R}$ and $K e^{Gt} = e^{\lambda t } K$ holds.)

    If we set $X_1 : = \ker ( K)$ and $X_0 := X / X_1$,
    then $K$ is the composition of a projection onto $X_0$
    and the induced map
    $X_0 \cong \mathrm{im}(K)$.
    Since $X_1 \supset \ker(P_\Omega\vert_X)$, $X_0$ must be contained in $P_\Omega(X)$,
    hence $X_0 \subset \mathrm{dom}(G)$, which, combined with the assumptions, implies that \ref{hypA2b} holds.
    Note also that, since eigenvalues appear in conjugate pairs, \eqref{eq:nonres_koopman_2} holds for $\overline{\lambda}$ as well; 
    it follows that \ref{hypA3b}-\ref{hypA4b} hold
    with $G_0$ as above.
    Moreover, \ref{hypA5b} holds trivially.

    We may apply Theorem~\ref{thm:sf} using Remark~\ref{remark:mfd_exist_uniq}, yielding a unique foliation tangent to $X_1$ at $0$, i.e., a submersion $\pi:X \supset U \to \mathbb{R}^2$ satisfying 
    \begin{equation}
        \pi \circ \varphi_t = 
        \begin{pmatrix}
            e^{t \, \mathrm{Re} \, \lambda} & -e^{t \, \mathrm{Im} \, \lambda} \\
            e^{t \, \mathrm{Im} \, \lambda} & e^{t \, \mathrm{Re} \, \lambda}
        \end{pmatrix}
        \pi,
        \label{eq:foli_exist}
    \end{equation}
    defined up to the usual conjugacy relationship \eqref{eq:uniqueness_sf} through elements  $\theta \in \mathrm{GL}(\mathbb{R}^2)$. 
    Transferring \eqref{eq:foli_exist} back via the isomorphism $ \varpi$ now gives two possibilities for a Koopman eigenfunction, one with eigenvalue $\lambda$ and one with $\overline{\lambda}$, through the choices of $\theta = \mathrm{id}_{\mathbb{R}^2}$ and $\theta = \begin{pmatrix}
        1 & 0 \\ 0 & -1
    \end{pmatrix}$, respectively. 
    These are defined up to complex multiplication and each represent an element of $\mathscr{K}^r(0)$.
    By Lemma~\ref{lemma:koopman_foli_1_1_new}, there cannot be more Koopman eigenfunctions associated to this foliation. 
    (In the $\lambda \in \mathbb{R}$ case, we have $\pi \circ \varphi_t = e^{\lambda t } \pi$ in place of \eqref{eq:foli_exist}, which is defined up to conjugacy by $\mathbb{R} \setminus \{0\}$, yielding a unique element of $\mathscr{K}^r(0)$.)

    The only source of ambiguity above was in selecting $Q$. 
    If $\lambda$ is an algebraically simple eigenvalue, then 
    there is only a single choice for $Q$, uniquely determining $X_1$.
    The above procedure thus yields unique elements of $\mathscr{K}^r(0)$ for $\lambda$ and $\overline{\lambda}$.\footnote{It should be noted that if we were to start with the $\overline{\lambda}$ eigenvalue, then its associated dual eigenvector on $\mathrm{im}(P_\Omega)$ is $ \overline{Q}$, which yields the same kernel $X_1$ of $K$ and hence determines the same foliation; the uniqueness claim remains unaffected.}
\end{proof}

\begin{corollary}[Uniqueness of Koopman eigenfunctions] \label{cor:Koop_uniq} 
    Let $\varphi$ be a local semiflow satisfying \ref{hypA1} on a real or complex Banach space $X$ (with $r = \omega$ if $X$ is complex).
    Suppose $t \mapsto D \varphi_t (0)$ is an analytic semigroup on $X$ with generator $G$ satisfying $\sigma(G) \subset \{ \mu \in \mathbb{C} \, |  \, \mathrm{Re} \, \mu < 0 \}$. 
    If $\psi$ is a $C^r_b$ Koopman eigenfunction with eigenvalue $\lambda$ at $0$
    such that there exists an integer $\ell < r$ for which
    \begin{equation}
        (\ell + 1) \sup \{ \mathrm{Re} \, \mu \, | \, \mu \in \sigma(G) \} < \mathrm{Re} \, \lambda
        \label{eq:Koop_uniq_ass}
    \end{equation}
    holds, then $\lambda \in n \sigma(G)$ for some $1 \leq n \leq \ell$. 
\end{corollary}

\begin{proof}
    The proof is the same as in \cite{kvalheim2021existence}, we reproduce it in Appendix~\ref{appendix:B} for the sake of completeness.
\end{proof}

The above uniqueness claim continues to hold if $\varphi$ is merely $C^1_b$ (in space), as shown in \cite{kvalheim2021existence} in the finite-dimensional setting.
Their proof carries over with very minor modifications to Banach spaces, 
hence we do not repeat it here.

\subsection{Unstable case} \label{sect:unstablecase}

As an aside from the majority of the preceding text,
the present section gives a brief outline of a possible route to obtain Koopman eigenfunctions for the case of unstable (hyperbolic) fixed points of semiflows.
We will show the existence of $C^0$ eigenfunctions under fairly mild assumptions, 
which, at least to the best of our knowledge, recovers the degree 
to which
the unstable scenario is developed in the finite dimensional case. 
Of course, in the finite dimensional case, or even that of bounded generators, $C^0$ eigenfunctions are readily provided by a direct application of the Hartman-Grobman theorem; but this does not necessarily extend
to semiflows with unbounded generators
(in cases where it does, e.g., the scalar reaction-diffusion equation \cite{lu1991hartman}, the Hartman-Grobman theorem has been proven via invariant foliations).

Here, we recall a result of \cite{chen1997invariant} on invariant manifolds and foliations for semiflows, 
which reduces the (unbounded) problem to finite dimensions.

\begin{theorem}[Theorem~1.1, \cite{chen1997invariant}] \label{thm:cht} 
Let $X$ be a Banach space admitting a $C^1$ cut-off function, 
and let $\varphi : \mathcal{D}^{\varphi} \to X$ be a jointly $C^0$ local semiflow of the form
\begin{displaymath}
    \varphi_t = L_t + N_t, \qquad t \geq 0,
\end{displaymath}
for $L_t \in \mathcal{L}(X)$ and $N_t : X \to X$ a locally Lipschitz map (on a neighborhood of $0$) satisfying $N_t(0) = 0$. 
Suppose there exists a constant $q > 0$ such that $\sup_{0 \leq t \leq q} \mathrm{Lip}(\varphi_t) < \infty$, and $\tau \in (0,q]$
for which $\varphi_\tau: X \to X$ is $C^1$ and
$X$ splits into closed $L_\tau$-invariant subspaces $X = X_0 \oplus X_1$ such that $L_\tau$ commutes with the associated projections $P_i: X \to X$, $i = 0,1$.
Denoting by $L_{\tau,i}:X_i \to X_i$ the restrictions, assume $L_{\tau,0} \in \mathrm{Aut}(X_0)$, and that there exist constants $\alpha_0 > \alpha_1 \geq 0$ and $C_0,C_1 \geq 1$ such that
\begin{align*}
     \Vert L_{\tau,0}^{-k} \widetilde{P}_0 \Vert &\leq C_0 \alpha_0^{-k}, \\
     \Vert L_{\tau,1}^{k} \widetilde{P}_1 \Vert &\leq C_1 \alpha_1^{k}, \qquad k \geq 0.
\end{align*} 

Then, there exists a locally invariant $C^1$ manifold $W_0$ tangent to $X_0$ at $0$,
and a locally invariant $C^0$ foliation $\{ \mathscr{L}_p \}_{p \in W_0}$ about $0$ (with $C^1$ leaves $\mathscr{L}_p$) tangent to $X_1$ at $0$.

\end{theorem}

\begin{proof}
    This is just a restatement of Theorem~1.1 of \cite{chen1997invariant}.
    The assumptions as we have stated imply that, once a $C^1$ cut-off function $\chi$ (with sufficiently small support) is applied to $N_t$, i.e., $\varphi_t$ is replaced by $\varphi_t^{\chi}(x) = L_t x + \chi(x) N_t(x)$, the hypotheses of the original theorem in \cite{chen1997invariant} are satisfied.
    Since $\varphi$ and $\varphi^\chi$ locally agree, it is clear then that we may draw conclusions locally.
\end{proof}

The conclusions of Theorem~\ref{thm:cht} amount to the existence of a $C^0$ projection map $\pi: O \to W_0$ on some open neighborhood $O$ of the origin in $X$, such that $\mathscr{L}_p \cap O = \pi^{-1}(p)$.
By local invariance of the foliation and $W_0$, $\pi$ satisfies 
\begin{equation}
    \pi \circ \varphi_t(x) = \varphi_t \vert_{W_0} \circ \pi(x), \qquad \text{for } (t,x) \in \mathcal{O},
    \label{eq:unstable_foli}
\end{equation}
with $\mathcal{O}$ as in \eqref{eq:Koopmandomain}.

If $X_0$ and hence $W_0$ can be chosen finite dimensional,
and if hyperbolicity of the fixed point $0$ of $\varphi_t \vert_{W_0}$ is assumed,
we may apply the Hartman-Grobman theorem to deduce the existence of $C^0$ Koopman eigenfunctions associated to the spectrum of $D \varphi_t (0) \vert_{X_0}$ via \eqref{eq:unstable_foli}. 

One instance in which this can always be achieved is if the generator $G$ of $L_t$ is sectorial, has
compact resolvent and its spectrum is disjoint from the imaginary axis.
Then, for any spectral subset of the form
\begin{displaymath}
    \Sigma = \{ \, \lambda \in \sigma(G) \, | \, \mathrm{Re} \, \lambda > c \, \},
\end{displaymath}
for some $c \in \mathbb{R}$ 
dissecting the spectrum with nonzero spectral gap,
the spectral projections $P_\Sigma$ and $(\mathrm{id}_X - P_\Sigma)$ yield an exponential dichotomy of the form required by Theorem~\ref{thm:cht}, with $X_0 := P_\Sigma X$ finite dimensional. 

\section{Example: the Navier-Stokes equations} \label{sect:example}

The results herein were posed with the 
intention that they would be   
applicable to (but not limited to) semilinear parabolic equations in the formulation of \cite{henry1981geometric}.
To give one example of the sort,
we consider the Navier-Stokes system, motivated by its importance in relation to Koopman theory.

Let us consider, then, 
a fluid enclosed within $\Omega \subset \mathbb{R}^d$, $d =2,3$, a bounded domain with piecewise smooth, rigid boundary $ \partial \Omega$.
The fluid obeys the Navier-Stokes equations,
\begin{subequations}
\begin{align}
    \partial_t v - \nu \Delta v + (v \cdot \nabla ) v + \nabla p &= f \quad &\text{in } \mathbb{R}^{\geq 0} \times \Omega, \\
    \nabla \cdot v &= 0 \quad &\text{in } \mathbb{R}^{\geq 0} \times  \Omega, \\
    v &= 0 \quad &\text{on } \mathbb{R}^{\geq 0} \times  \partial \Omega, \\
    v(0, \cdot ) &= \hat{v} \quad &\text{in } \Omega,
\end{align}\label{eq:NSE}\end{subequations}
where $\nu$ is the kinematic viscosity; $v=v(t,x)$ and $p=p(t,x)$ are the sought velocity and pressure fields, respectively; $f=f(x)$ is a time-independent forcing term; and $\hat{v}=\hat{v}(x)$ is an initial condition.
We suppose the existence of a smooth steady-state solution, $(U,P)$,
which will shortly be moved to the origin, and will proceed to serve as the base fixed point to which we 
apply the machinery of sections \ref{sect:statement} and \ref{sect:koopman}.

Consider the space
\begin{equation}
    V = \overline{ \Big\{ \, u \in C^\infty_c (\Omega; \mathbb{R}^d) \; \big\vert \; \mathrm{div} \, u = 0 \, \Big\} }^{H^1(\Omega; \mathbb{R}^d)}
    \label{eq:V}
\end{equation}
with norm
\begin{displaymath}
    \Vert u \Vert_V^2 = \int_\Omega |\nabla u (x)|^2 dx,
\end{displaymath} 
and define the operators
\begin{equation}
    A_0 u = \nu P_L \Delta u, \qquad u \in \mathrm{dom}(A_0),
    \label{eq:AU}
\end{equation} 
with $\mathrm{dom}(A_0) = V \cap H^3(\Omega; \mathbb{R}^d)$ and $P_L$ denoting the Leray projection from $H^1(\Omega; \mathbb{R}^d)$ onto $V$; and
\begin{equation}
    B (u,w) = P_L[(u \cdot \nabla)w], \qquad u,w \in \mathrm{dom}(A_0).
    \label{eq:B}
\end{equation} 

It is well known that, under the present hypotheses, 
\eqref{eq:NSE} may be equivalently written as an evolution equation \cite{Fujita1964,temam2001navier,Temam1995NSEandNFA,foias2001}, 
\begin{subequations}
  \begin{gather}
    \frac{d}{dt} u = A_U u - B(u,u), \\
    u(0) = \hat{u},
\end{gather}\label{eq:NSE_evolution}\end{subequations}
which we shall consider on $V$.
Here, $u(t) = v(t, \cdot) - U$ is the perturbation velocity field (and similarly, $\hat{u} = \hat{v} - U$ is the corresponding initial condition); and $A_U : \mathrm{dom}(A_U) = \mathrm{dom}(A_0) \to V$,
\begin{displaymath}
    A_U = A_0 - B(U, \cdot) - B( \cdot, U),
\end{displaymath}
is a sectorial operator with compact resolvent \cite{yudovich1989}.

It is also well known that \eqref{eq:NSE_evolution} generates a local semiflow on $V$ (which extends to define a global semiflow if $d = 2$), $\varphi : \mathcal{D}^\varphi \to V$ that is separately continuous on its domain and jointly $C^{\infty}$ over $\mathcal{D}^{\varphi} \cap (\mathbb{R}^{> 0} \times V)$ \cite{WEISSLER1979,Weissler1980TheNI,marsden1976hopf,robinson2001,buza2023spectral}. 
Moreover, the joint continuity of $(t,v_0) \mapsto D^i (\varphi_t  \vert_{V_0}) (v_0)$, $i \geq 0$, extends to $(t,v_0) = (0,0)$, for any $V_0 \subset \mathrm{dom}(A_U)$.
 While fairly obvious, the latter claim  gets a little technical, its proof is hence deferred to Appendix~\ref{appendix:A}.

We remark that
spectral stability, i.e., $\sigma(A_U) \subset \{ \mu \in \mathbb{C} \; | \; \mathrm{Re} \, \mu < 0 \}$,
implies Lyapunov stability in the norm topology of $V$ (see Theorem 2.3 of Chapter 2 and the discussion on pages 113-114 of \cite{yudovich1989}).
With Lyapunov stability in hand, we may construct a neighborhood $O \subset V$ of $0$ on which $\varphi_t$ is $C^\infty_b$ for each $t \geq 0$ (see Appendix~\ref{appendix:A}).

We have therefore established all the requisites of Theorem~\ref{thm:sf} for splittings of the form $V_0 \oplus V_1$, such that $V_0 \subset \mathrm{dom}(A_U)$, and with respect to which $A_U$ satisfies the nonresonance conditions \ref{hypA3b}-\ref{hypA4b}.
Splittings as such can be readily supplied via spectral projections $P_\Sigma$ 
according to nonresonant spectral subsets $\Sigma \subset \sigma(A_U)$; the existence of a $C^\infty$ invariant foliation about $0$ tangent to $\ker ( P_\Sigma  )$ ensues (with uniqueness in a sufficiently smooth subclass).
Note that the reduced dynamics produced by Theorem~\ref{thm:sf} will take the form of a jointly $C^\infty$ flow $\vartheta$, since $\mathrm{im}(P_\Sigma)$ is finite dimensional.

Similarly,
Corollary~\ref{cor:Koop_exist} guarantees the
existence of a principal Koopman eigenfunction at $0$ for every $\lambda \in \sigma(A_U)$ satisfying the nonresonance conditions \eqref{eq:nonres_koopman_2} 
(\eqref{eq:nonres_koopman_1} can obviously be satisfied by choosing $\ell$ large enough). 
If $\lambda$ is algebraically simple, then the obtained eigenfunction is
unique in the appropriate sense.

An important subcase in which the nonresonance conditions are automatically satisfied is that of the least stable mode. 
Otherwise, the nonresonance conditions have to be checked numerically on an example-by-example basis, 
for which there exist rigorous methods providing error bounds on the approximation 
in the case of operators on bounded domains \cite{boffi2000problem} (the case of unbounded domains gets much more finicky, see e.g.\ \cite{colbrook2022foundations}).

If the steady state $(U,P)$ is unstable, none of the preceding theory applies except for that of Section~\ref{sect:unstablecase}, but with its use, we do obtain $C^0$ Koopman eigenfunctions for all eigenvalues without the need to satisfy nonresonance conditions.

\section{Proof of the main results}
\label{sect:proof}

In this section, we prove Theorems \ref{thm:maps} and \ref{thm:sf}.
We begin by recalling some 
procedural standards forming the core of the proof, following \cite{Cabre2003a}.

\begin{remark}[Adapted norms] \label{remark:adapt}

    For any $\varepsilon > 0$,
    It is possible to construct a new norm $| \cdot |_{\mathrm{a}}$ on $X$, equivalent to the original one, according to which 
    \begin{align*} 
        &\Vert A \Vert_{\mathcal{L}(X)} \leq \rho_A + \varepsilon,
        & &\Vert A_0^{-1} \Vert_{\mathcal{L}(X_0)} \leq \rho_{A_0^{-1}} + \varepsilon, 
    \end{align*} 
    hold, where $\rho_A = \sup \{ |z| \, | \, z \in \sigma(A) \}$ denotes the spectral radius. 
    For a proof of this claim, see \cite{Cabre2003a}, Proposition A.1.

    In such a norm, \ref{hyp3:ell} implies
    \begin{equation} 
        \Vert A_0^{-1} \Vert_{\mathcal{L}(X_0)} \Vert A \Vert_{\mathcal{L}(X)}^{\ell + 1} < 1,
        \label{eq:hyp3_adapt}
    \end{equation}
    which is the form \ref{hyp3:ell} will be used in later on.  

    We henceforth assume without further note that 
    the norm adaptation
    has been carried out
    and dispose of the subscript '${\mathrm{a}}$'.
\end{remark}

\begin{remark}
    Rather than considering a small neighborhood $U$ as in the statement of Theorem~\ref{thm:maps}, we instead assume in the following sequence of lemmas that $\Vert \phi \Vert_{C^r_b(B_1^{X};X)}$ is sufficiently small.
    This is an equivalent approach and is common practice in the literature, but we nonetheless address this equivalence towards the end of the proof. 
\end{remark}

Note also that with $\Vert \phi \Vert_{C^r_b(B_1^{X};X)}$ small, we have by Remark~\ref{remark:adapt} that $f$ leaves the unit ball invariant, i.e.\ $f(B_1^X) \subset B_1^X$, which will be assumed throughout.

\begin{remark} \label{remark:linearpart}
    From \eqref{eq:map_foliation_diag}, we necessarily have that $g(0) = 0$.
At linear order, \eqref{eq:map_foliation_diag} reads
\begin{displaymath}
    Dg(0) D\pi(0) = D \pi(0) A, 
\end{displaymath}
which implies that $A$ leaves $\ker ( D\pi (0) )$ invariant.
On account of the assumed form of $A$ and  requirement \eqref{eq:Dpi_Aut}, we must have $\ker ( D\pi (0) ) = X_1$. 
Moreover,
\begin{align*}
    Dg(0) &= D \pi(0)  A \vert_{X_0}  ( D\pi(0) \vert_{X_0} )^{-1} \\
        &= D\pi(0) \vert_{X_0} A_0 ( D\pi(0) \vert_{X_0} )^{-1}.
\end{align*}
Throughout the existence part of the proof, we proceed with the choice $D\pi(0) \vert_{X_0} = \mathrm{id}_{X_0}$ (and hence $Dg(0) = A_0$), but note that Remark~\ref{remark:diffeo} will then give a solution with $D\pi(0) \vert_{X_0} = C$ for any choice of $C \in \mathrm{Aut}(X_0)$ while retaining the polynomial order of $g$. 
\end{remark}

We seek the foliation map $\pi$ in \ref{eq:map_foliation_diag} as 
\begin{equation}
    \pi = \pi^\leq + \pi^>,
    \label{eq:pi_decomp}
\end{equation}
for $\pi^\leq$ a lower order polynomial expression and $\pi^>$ the remainder.
To deal with the polynomial part $\pi^\leq$, we recall some terminology about multilinear maps.

Let $X_1,\ldots , X_n$, $Y$ and $Z$ be Banach spaces.  
By convention, we endow product spaces with the norm $|(x_1,\ldots,x_n)|_{X_1 \times \ldots \times X_n} = \sup \{|x_1|_{X_1},\ldots,|x_n|_{X_n}\}$.
We denote by $\mathrm{M}(X_1,\ldots,X_n;Y)$ the Banach space of bounded multilinear maps 
\begin{displaymath}
    X_1 \times \ldots \times X_n \longrightarrow Y
\end{displaymath}
equipped with the norm 
\begin{displaymath}
    \Vert N \Vert_{\mathrm{M}(X_1,\ldots,X_n;Y)} = \sup \big\{ | N(x_1, \ldots, x_n) |_Y \; \big| \; |(x_1, \ldots, x_n)|_{X_1 \times \ldots \times X_n} < 1 \big\}.
\end{displaymath}

For $1 \leq k \leq n$ and bounded linear maps $L \in \mathcal{L}(Z,X_k)$ and $K \in \mathcal{L}(Y,Z)$, consider the precomposition operator $$r^k_L: \mathrm{M}(X_1,\ldots,X_k,\ldots,X_n;Y) \to  \mathrm{M}(X_1,\ldots,Z,\ldots,X_n;Y)$$ given by
\begin{equation}
    r^k_L N(x_1,\ldots,z,\ldots,x_n) = N(x_1,\ldots,Lz,\ldots,x_n).
    \label{eq:r_k}
\end{equation}
along with $l_K : \mathrm{M}(X_1,\ldots,X_n;Y) \to  \mathrm{M}(X_1,\ldots,X_n;Z)$ defined as
\begin{equation}
    l_K N (x_1,,\ldots, x_n) = K N (x_1,\ldots, x_n).
    \label{eq:l}
\end{equation}
Both maps $L \mapsto r^k_L$ and $K \mapsto l_K$ are linear.
Moreover, for $K' \in \mathcal{L}(Z,Z')$ and $L'\in \mathcal{L}(Z',Z)$, we have that $l_{K'} l_{K} = l_{K'K}$ and $ r^k_{L'} r^k_L= r^k_{L L'}$.
The inequality
\begin{equation}
   \Vert r^k_L N \Vert_{\mathrm{M}(X_1,\ldots,Z,\ldots,X_n;Y)} \leq \Vert N \Vert_{\mathrm{M}(X_1,\ldots,X_k,\ldots,X_n;Y)} \Vert L \Vert_{\mathcal{L}(Z,X_k)}
   \label{eq:r_cont}
\end{equation}
shows that both $r^k_L$ and $L \mapsto r^k_L$ are bounded and hence continuous by linearity; similarly for $l_K$ and $K \mapsto l_K$. 

When $K \in \mathcal{L}(Y)$, 
the above properties imply
that $$K \mapsto l_K : \mathcal{L}(Y) \to\mathcal{L}( \mathrm{M}(X_1,\ldots X_n;Y))$$ is a Banach algebra homomorphism with $l_{\mathrm{id}_{Y}} = \mathrm{id}_{\mathrm{M}(X_1,\ldots X_n;Y)}$; similarly $r^k$ is an anti-homomorphism from $\mathcal{L}(X_k)$ to $\mathcal{L}( \mathrm{M}(X_1,\ldots X_n;Y))$.
The spectra of $l_K$ and $r^k_L$ hence satisfy
\begin{subequations}
\begin{align}
    &\sigma (l_K; \mathcal{L}(\mathrm{M}(X_1,\ldots X_n;Y))) \subset \sigma (K;\mathcal{L}(Y)), \\
    &\sigma (r^k_L; \mathcal{L}(\mathrm{M}(X_1,\ldots X_n;Y)) ) \subset \sigma (L; \mathcal{L}(X_k)). 
\end{align}\label{eq:spectrum_lk_rk}
\end{subequations}

We use the shorthand notation $\mathrm{M}_n(X;Y)$ for the space $\mathrm{M}(X,\ldots,X;Y)$, where $X$ appears $n \geq 0$ times, with the convention that $\mathrm{M}_0(X;Y) = Y$.
Given two closed linear subspaces $X_0,X_1 \subset X$ such that $X = X_0 \oplus X_1$ and $\alpha \in \{0,1\}^n$, we denote by $\mathrm{M}_\alpha(X;Y)$ the space $\mathrm{M}(X_{\alpha(1)},\ldots,X_{\alpha(n)};Y)$.
With $\alpha$ as above and $L$ a bounded linear map,
let
\begin{displaymath}
    r_L^{\alpha} := (r^1_L)^{\alpha(1)} \cdots (r^n_L)^{\alpha(n)},
\end{displaymath}
with the convention that $(r^k_L)^0 = \mathrm{id}$. 
Denote further $\overline{\alpha}: = (1,\ldots,1) - \alpha$, and $r_L := r_L^{(1,\ldots,1)}$.

With these conventions, we have 
\begin{equation}
    r^\alpha_{L + L'} = \sum_{\substack{\beta,\gamma \in \{0,1\}^n \\ \beta + \gamma = \alpha}} r^\gamma_L r^\beta_{L'},
    \label{eq:alpha_beta_gamma_sum}
\end{equation}
which will be used repeatedly in what follows.

\begin{lemma} 
    Let $X$ be as in \ref{hyp1:X}.
    Let $Y$ be another Banach space and let $n \geq 1$.
    Then, the map 
    \begin{align}
        \eta  : \: \mathrm{M}_n(X;Y) &\longrightarrow \prod_{\alpha \in \{ 0,1\}^n} \mathrm{M}_\alpha(X;Y) \nonumber \\
        N  &\mapsto (r_{\imath_0}^{\overline{\alpha}} r_{\imath_1}^{\alpha}N)_{\alpha \in \{ 0,1\}^n}
        \label{eq:eta}
    \end{align}
    is an isomorphism of Banach spaces. 
\end{lemma}

\begin{proof}
    Linearity of $\eta$ is clear from its definition.
    Inclusion maps from the subspace topology are continuous -- it follows from \eqref{eq:r_cont} that so is $r^k_{\imath_j}$
    and hence \eqref{eq:eta} is well-defined in that it has image within the space of bounded maps. 
    Since $r_{\imath_0}^{\overline{\alpha}} r_{\imath_1}^{\alpha}$ is bounded
    for any $\alpha \in \{ 0,1\}^n$, we have that $\eta$ is bounded.

    We now construct an inverse. 
    Since the $X_i$ are closed, the projections  $\widetilde{P}_i : X \to X_i$ with restricted range are continuous. 
    It follows via \eqref{eq:r_cont} that,
    for $(N_\alpha)_{\alpha \in \{ 0,1 \}^n} \in \prod_{\alpha \in \{ 0,1\}^n} \mathrm{M}_\alpha(X;Y)$, 
    the map
    \begin{equation}
        \xi: (N_{\alpha})_{\alpha \in \{ 0,1\}^n} \mapsto \sum_{\alpha \in \{0,1 \}^n} r^{\overline{\alpha}}_{\widetilde{P}_0} r^{\alpha}_{\widetilde{P}_1} N_\alpha.
        \label{eq:xi}
    \end{equation} 
    takes values in $\mathrm{M}_n(X;Y)$.

    We show that $\xi \circ \eta = \mathrm{id}_{\mathrm{M}_n(X;Y)}$.
    Indeed, 
    \begin{align*}
        \xi \circ \eta &= \sum_{\alpha \in \{ 0,1\}^n} r^{\overline{\alpha}}_{\widetilde{P}_0} r^{\alpha}_{\widetilde{P}_1} r_{\imath_0}^{\overline{\alpha}} r_{\imath_1}^{\alpha}  \\
        &= \sum_{\alpha \in \{ 0,1\}^n} r_{\imath_0 \widetilde{P}_0}^{\overline{\alpha}} r_{\imath_1 \widetilde{P}_1}^{\alpha} \\
        &= \sum_{\alpha \in \{ 0,1\}^n} r_{P_0}^{\overline{\alpha}} r_{P_1}^{\alpha} \\
        &= (r^1_{P_0+P_1}) \cdots (r^n_{P_0+P_1} ) \\
        &= \mathrm{id}_{\mathrm{M}_n(X;Y)},
    \end{align*}
    where the penultimate equality used \eqref{eq:alpha_beta_gamma_sum}.
    Conversely,
    \begin{align*}
        \eta \circ \xi (N_{\alpha})_{\alpha \in \{ 0,1\}^n} &= \left( r_{\imath_0}^{\overline{\beta}} r_{\imath_1}^{\beta} \sum_{\alpha \in \{0,1 \}^n} r^{\overline{\alpha}}_{\widetilde{P}_0} r^{\alpha}_{\widetilde{P}_1} N_\alpha \right)_{\beta \in \{0,1\}^n} \\ 
        &= \left( r^{\overline{\beta}}_{\widetilde{P}_0 \imath_0} r^{\beta}_{\widetilde{P}_1 \imath_1} N_\beta  \right)_{\beta \in \{0,1\}^n} \\
        &= \left( r^{\overline{\beta}}_{\mathrm{id}_{X_0}} r^{\beta}_{\mathrm{id}_{X_1}} N_\beta  \right)_{\beta \in \{0,1\}^n} \\
        &= \left(  N_\beta  \right)_{\beta \in \{0,1\}^n}.
    \end{align*}
    Here, the second equality follows from $\widetilde{P}_j \imath_k =0$ whenever $j \neq k$.  
    Finally, that $\xi$ is bounded can be seen either directly from \eqref{eq:xi} or the open mapping theorem. 
\end{proof}

Next, we recall a classical result about spectra,
which will be of use in the lemma to follow.

\begin{theorem}[Theorem~11.23, \cite{rudin1991functional}] \label{thm:spectrum_BA_comm}
Let $a$ and $b$ be two commuting elements of a Banach algebra $\mathcal{A}$. 
Then
\begin{displaymath}
    \sigma(ab;\mathcal{A}) \subset \sigma (a;\mathcal{A}) \sigma (b;\mathcal{A}).
\end{displaymath}
\end{theorem}

Note that $l_{L}$, $r^k_{L'}$ and $r^j_{L''}$ all commute with each other if $j \neq k$,
and so Theorem~\ref{thm:spectrum_BA_comm} will be applicable to operators of the form $l_L r^{\overline{\alpha}}_{L'} r^\alpha_{L''}$.

\begin{lemma} \label{lemma:approx} 

Assume hypotheses \ref{hyp1:X}, \ref{hyp2:f}, \ref{hyp4:spec1} and let $\ell \leq r$ be an integer.

Then, we can find 
polynomials 
$\pi^{\leq} = \sum_{n = 1}^\ell \pi_n$
and
$g = \sum_{n = 1}^\ell g_n$,
where $\pi_n$ and $g_n$ are homogeneous 
polynomials of degree $n$,
such that
\begin{equation}
    g \circ \pi^\leq(x) = \pi^\leq \circ f (x) + o(|x|^\ell), \qquad x \in X
    \label{eq:approx_lemma}
\end{equation}
as $x \to 0$, with 
 $\pi_1 = \widetilde{P}_0$
and $g_1 = A_0$. 

If \ref{hyp5:spec2} holds for some $m \leq \ell$, then the $g_n$ with $n \geq m$ can be chosen $0$.

Moreover, the higher order terms can be estimated by $\phi$, the nonlinear part of $f$,
\begin{subequations}\label{eq:gi_pi_bound}
\begin{align}
    & \left\Vert g_n \right\Vert_{C^r_b(B_\rho^{X_0};X_0)} \leq C \left\Vert \phi \right\Vert_{C^r_b(B_1^{X};X)},  &2 \leq n \leq \ell, \label{eq:gi_bound} \\
    &\left\Vert \pi_n \right\Vert_{C^r_b(B_\rho^{X};X_0)} \leq C \left\Vert \phi \right\Vert_{C^r_b(B_1^{X};X)},  &2 \leq n \leq \ell, \label{eq:pi_bound}
\end{align}\end{subequations}
on balls of some finite radius $\rho > 0$.

If $X$  in \ref{hyp1:X} is real, then so are the resulting polynomials $g$ and $\pi^\leq$.

\end{lemma}

\begin{proof} 
We obtain equations for each polynomial order $1 \leq n \leq \ell$ by differentiating  
\begin{displaymath}
    g \circ \pi = \pi \circ f
\end{displaymath}
$n$ times and equating at $0$.
Once a solution is found at each order, Taylor's theorem yields \eqref{eq:approx_lemma}.
Order $n = 1$ was discussed in Remark~\ref{remark:linearpart},
taking $\pi_1 = \widetilde{P}_0$, $g_1 = A_0$ as in the statement is a suitable choice. 

We proceed inductively and assume that $\pi^{\leq}$ and $g$ have been found up to order $n-1$. 
At order $n$, the above procedure yields
\begin{equation}
   g_n  (\pi_1 [\, \cdot \,],\ldots, \pi_1 [\, \cdot \,]) + g_1 \pi_n - \pi_n (A [\, \cdot \,],\ldots,A [\, \cdot \,]) = \Gamma_n,
   \label{eq:order_n}
\end{equation}
where $\Gamma_n$ is composed of lower order terms: 
\begin{multline}
    \Gamma_n = \sum_{i=2}^{n-1} \sum_{\substack{j \in \mathbb{N}^i \\ |j| = n}} C_{j} \Big( \pi_i (D^{j_1}f(0) [\, \cdot \,],\ldots,D^{j_i}f(0) [\, \cdot \,])-g_i(\pi_{j_1} [\, \cdot \,],\ldots,\pi_{j_i} [\, \cdot \,]) \Big) \\ + \pi_1 D^n f(0)
    \label{eq:Gamma_n}
\end{multline}
for some coefficients $C_{j}$ resulting from Faà di Bruno's formula.
In particular, since the lower order terms $\pi_i$ and $g_i$ are symmetric by the inductive hypothesis, so is $\Gamma_n$.
We need to show that \eqref{eq:order_n} is solvable with $\pi_n$ and $g_n$ symmetric.

Applying the isomorphism $\eta$ from \eqref{eq:eta} to \eqref{eq:order_n}, we obtain $2^n$ equations, one for each choice of $\alpha \in \{0,1\}^n$:  
\begin{equation}
   r_{\imath_0}^{\overline{\alpha}} r_{\imath_1}^{\alpha} r_{\widetilde{P}_0} g_n + l_{A_0} \pi_\alpha - r_{\imath_0}^{\overline{\alpha}} r_{\imath_1}^{\alpha} r_{A} \pi_n = \Gamma_\alpha.
   \label{eq:order_n_projected}
\end{equation}
Here, $\Gamma_\alpha$ stands for $r_{\imath_0}^{\overline{\alpha}} r_{\imath_1}^{\alpha} \Gamma_n$, and similarly for $\pi_\alpha$.
Using $A \imath_0 = \imath_0 A_0 + \imath_1 B$, $A \imath_1 = \imath_1 A_1 $ and \eqref{eq:alpha_beta_gamma_sum}, \eqref{eq:order_n_projected} leads to
\begin{displaymath} 
   r_{\widetilde{P}_0 \imath_0}^{\overline{\alpha}} r_{\widetilde{P}_0 \imath_1}^{\alpha}  g_n + l_{A_0} \pi_\alpha -   r_{\imath_1 A_1 }^{{\alpha}} \sum_{\substack{\beta,\gamma \in \{0,1\}^n \\ \beta + \gamma = \overline{\alpha}}} r_{\imath_0 A_0 }^\beta r_{\imath_1 B }^{\gamma}  \pi_n = \Gamma_\alpha.  
\end{displaymath}
Rearranging the inclusion terms yields
\begin{equation}
    \delta_{(0,\ldots,0)}(\alpha)  g_n + \left( l_{A_0}  -   r_{ A_0 }^{\overline{\alpha}} r_{ A_1 }^{\alpha} \right) \pi_\alpha  = \Gamma_\alpha +   r_{ A_1 }^{{\alpha}} \sum_{\substack{\beta,\gamma \in \{0,1\}^n \\ \beta + \gamma = \overline{\alpha}; \, |{\beta}| < |\overline{\alpha}|}} r_{ A_0 }^\beta r_{ B }^{\gamma}  \pi_{\overline{\beta}}, 
    \label{eq:order_n_final}
\end{equation}
where $\delta$ is the Kronecker symbol given by $\delta_\beta(\alpha) = 1$ if $\alpha = \beta$ and $\delta_\beta(\alpha) = 0$ otherwise.
Noting that $|{\beta}| < |\overline{\alpha}|$ is equivalent to $|\overline{\beta}| > |\alpha|$, we observe that
equation \eqref{eq:order_n_final} can be solved in decreasing order according to $|\alpha|$. 

Starting from the last component of \eqref{eq:order_n_final}, with $\alpha = (1,\ldots,1)$, we have
\begin{equation}
     \left( l_{A_0}  -   r_{A_1} \right) \pi_{(1,\ldots,1)}  = \Gamma_{(1,\ldots,1)}.
    \label{eq:order_n_eq111}
\end{equation}
Given \ref{hyp4:spec1}, we have via Theorem~\ref{thm:spectrum_BA_comm} and \eqref{eq:spectrum_lk_rk} that
\begin{displaymath}
    \sigma (l_{A_0^{-1}} r_{A_1}) \subset \sigma(A_0^{-1}) \sigma(A_1)^n 
\end{displaymath}
does not contain $1$ as an element, thus
\begin{displaymath}
    l_{A_0}  -   r_{A_1} = l_{A_0} \left(  \mathrm{id}_{\mathrm{M}_n(X_1;X_0)} -  l_{A_0^{-1}} r_{A_1} \right) 
\end{displaymath}
is invertible with bounded inverse and hence \eqref{eq:order_n_eq111} is solvable uniquely and boundedly for $\pi_{(1,\ldots,1)}$.

Proceeding by induction over $|\alpha|$, the remainder of \eqref{eq:order_n_final} with $|\alpha| > 0$ may be solved analogously.
For $0<j <n$ an integer, assume that all $\{\pi_\mu\}_{\mu \in \{0,1\}^n}$ with $j <|\mu| \leq n$ have been computed.
Then, the right hand sides of \eqref{eq:order_n_final} for all $\alpha \in \{0,1\}^n$ with $|\alpha| = j$ are available; equations of the form
\begin{displaymath}
    \left( l_{A_0}  -   r_{ A_0 }^{\overline{\alpha}} r_{ A_1 }^{\alpha} \right) \pi_\alpha  = \widetilde{\Gamma}_\alpha
\end{displaymath}
are to be solved. 
Hypothesis \ref{hyp4:spec1}
guarantees that all operations of the form
\begin{displaymath}
    l_{A_0} \left(  \mathrm{id}_{\mathrm{M}_\alpha(X;X_0)} -  l_{A_0^{-1}} r^{\overline{\alpha}}_{A_0} r^{\alpha}_{A_1} \right)
\end{displaymath}
are invertible boundedly,
as
\begin{displaymath}
    \sigma (l_{A_0^{-1}} r^{\overline{\alpha}}_{A_0} r^{\alpha}_{A_1}) \subset \sigma(A_0^{-1}) \sigma(A_0)^{n-j} \sigma(A_1)^{j}. 
\end{displaymath}

The only remaining component of \eqref{eq:order_n_final} is the one with $\alpha = (0,\ldots,0)$ containing $g_n$,
\begin{equation}
    g_n + \left( l_{A_0}  -   r_{A_0} \right) \pi_{(0,\ldots,0)}  = \widetilde{\Gamma}_{(0,\ldots,0)},
    \label{eq:order_n_eq1}
\end{equation}
where $\widetilde{\Gamma}_{(0,\ldots,0)}$ is known.
Equation~\eqref{eq:order_n_eq1} is always solvable by taking $g_n = \widetilde{\Gamma}_{(0,\ldots,0)}$ and $\pi_{(0,\ldots,0)} = 0$.
If \ref{hyp5:spec2} holds at $n$, we have once more that $(l_{A_0}  -   r_{A_0})$ is invertible and hence $g_n = 0$, $\pi_{(0,\ldots,0)} = (l_{A_0}  -   r_{A_0})^{-1}\widetilde{\Gamma}_{(0,\ldots,0)}$ is also a solution to \eqref{eq:order_n_eq1}.

Using the inverse map from \eqref{eq:xi}, we obtain the solution $\pi_n = \xi (\pi_\alpha)_{\alpha \in \{0,1\}^n}$ to \eqref{eq:order_n}.
It remains to show symmetry of $\pi_n$, which proceeds exactly as in \cite{Cabre2003a}. 
We may assume that $g_n$ has been chosen to be symmetric. 
Note that once $g_n$ is fixed, the solution $\pi_n$ obtained above is the unique one solving \eqref{eq:order_n}.
Since $l_{A_0} -r_A$ commutes with permutations and $\Gamma_n$ is invariant under permutations, we have by uniqueness that $\pi_n$ is also invariant under permutations.

To see \eqref{eq:gi_pi_bound}, we proceed once more via induction on the polynomial order $n$.
At $n=2$, we have by \eqref{eq:Gamma_n} that
\begin{displaymath}
  \Vert  \Gamma_2 \Vert_{\mathrm{M}_2(X;X_0)} \leq C \Vert \phi \Vert_{C^r_b(B_1^{X};X)},
\end{displaymath}
and through \eqref{eq:order_n_final} that
\begin{align*}
    &\Vert g_2 \Vert_{\mathrm{M}_2(X_0;X_0)} \leq C \Vert \phi \Vert_{C^r_b(B_1^{X};X)}, \\
    &\Vert \pi_2 \Vert_{\mathrm{M}_2(X;X_0)} \leq C \Vert \phi \Vert_{C^r_b(B_1^{X};X)}.
\end{align*}
Assuming these hold for orders up to $n-1$, \eqref{eq:Gamma_n} gives once more that 
$
  \Vert  \Gamma_n \Vert_{\mathrm{M}_n(X;X_0)} \\ \leq C \Vert \phi \Vert_{C^r_b(B_1^{X};X)},
$
and \eqref{eq:order_n_final} shows the same for $g_n$ and $\pi_n$.
Now \eqref{eq:gi_pi_bound} follows from 
\begin{displaymath}
    \Vert g_i \Vert_{C^r_b(B_\rho^{X_0};X_0)}  \leq  C \Vert g_i \Vert_{\mathrm{M}_i(X_0;X_0)},
\end{displaymath}
and similarly for $\pi_i$.

Finally, if $X$ is a real Banach space, we use the device of complexification to transform 
\eqref{eq:order_n_final} into an equation over $\mathrm{M}_\alpha (X;X_0)^c$, 
so that the we can make sense of the spectra of the relevant operators.
We note, as in \cite{Cabre2003a}, that the operators $l_K$ and $r_K^\alpha$ are well-behaved under complexification, i.e., $l_{K^c} = (l_K)^c$ and $r_{K^c}^\alpha = (r_K^\alpha)^c$, so that they preserve the real subspace of $\mathrm{M}_\alpha (X;X_0)^c$. 
The claim now follows by an inductive argument (over both $n$ and $|\alpha|$), for if lower order terms are real, then so are the solutions $g_n$ and $\pi_\alpha$ at each step. 
\end{proof}

\begin{remark} \label{remark:uniqueness_approxlemma}
    It follows from the proof of Lemma~\ref{lemma:approx} that, given a degree-$\ell$ polynomial on $X_0$, $\pi^0 = \sum_{n=2}^\ell \pi^0_n$, in place of $\pi_{(0,\ldots,0)}$ in \eqref{eq:order_n_eq1}, the remainder of the $g_n$ and $\pi_\alpha$ are uniquely determined.
    In particular, if $(g,\pi)$ and $(\tilde{g},\tilde{\pi})$ are two $C^\ell_b$ solutions of \eqref{eq:map_foliation_diag} such that\footnote{Here,  $(\cdot)^{\leq}$ refers to the $\ell$-jet (or $\ell$-Taylor expansion) of $(\cdot)$.} 
    $\pi^{\leq} \vert_{X_0} = \tilde{\pi}^{\leq} \vert_{X_0}$, then $\pi^{\leq} = \tilde{\pi}^{\leq}$ and $g^{\leq} = \tilde{g}^{\leq}$.
\end{remark}

The following segment proceeds completely analogously to \cite{Cabre2003a}.
Write $\psi:= \sum_{i=2}^\ell g_i$ so that $g = A_0 + \psi$.
We are looking to solve the remainder of 
\eqref{eq:map_foliation_diag} with $\pi^{\leq}$ available, i.e.,
\begin{equation}
    A_0 \pi^> - \pi^> \circ f = - \psi \circ (\pi^{\leq} + \pi^>) + \pi^{\leq} \circ f - A_0 \pi^{\leq},
    \label{eq:maineq_split}
\end{equation}
for the unknown $\pi^>$.
The left hand side of \eqref{eq:maineq_split} is linear, 
and hence we shall deal with this part first.

Equation \eqref{eq:maineq_split} will be solved over 
the same Banach spaces $\Gamma_{s,l}$ that were used in \cite{Cabre2003a}. 
Namely, for two Banach spaces $Y$, $Z$; $s \in \mathbb{N} \cup \{ \omega \}$ and $l \in \mathbb{N}$, $s \geq l$, we consider    
\begin{multline}
    \Gamma_{s,l}(B^{Y}_1;Z) = \bigg\{   h \in C^s_b(B^{Y}_1;Z) \; \Big| \; D^kh(0) = 0 \text{ for } 0 \leq k \leq l, \\
    \sup_{y \in B^{Y}_1} \frac{\Vert D^l h(y) \Vert_{\mathrm{M}_{l}(Y;Z)} }{|y|} < \infty \bigg\}, 
    \label{eq:Gammasldef}
\end{multline}
which forms a Banach space when equipped 
with the norm (see \cite{Cabre2003a})
\begin{equation}
    \Vert h \Vert_{\Gamma_{s,l}(B^{Y}_1;Z)} = \max \left\{ \Vert h \Vert_{C^s_b(B^{Y}_1;Z)} , \sup_{y \in B^{Y}_1} \frac{\Vert D^l h(y) \Vert_{\mathrm{M}_{l}(Y;Z)} }{|y|} \right\}
    \label{eq:Gammslnorm}
\end{equation}
if $s \in \mathbb{N}$, and
\begin{displaymath}
    \Vert h \Vert_{\Gamma_{\omega,l}(B^{Y}_1;Z)} = \Vert D^{l+1} h \Vert_{C^0_b(B^{Y}_1;\mathrm{M}_{l+1}(Y;Z))}  
\end{displaymath}
if $s = \omega$. 
If the spaces $Y$ and $Z$ are clear from the context, we shall use the shorthand notation $\Gamma_{s,l}$ for $\Gamma_{s,l}(B^{Y}_1;Z)$.

We remark that if $s > l$, the last condition in \eqref{eq:Gammasldef}
is superfluous, as
\begin{equation}
    \Vert D^l h(y) \Vert_{\mathrm{M}_{l}(Y;Z)} \leq |y| \int_0^1 \Vert D^{l+1} h(sy) \Vert_{\mathrm{M}_{l+1}(Y;Z)} ds,
    \label{eq:lastpartgamma}
\end{equation}
and the norm \eqref{eq:Gammslnorm} is equivalent to $\Vert \cdot \Vert_{C^s_b(B_1^Y;Z)}$.
For the existence result, we use $\Gamma_{r,\ell}$ with $r> \ell$; the $s= l$ version is needed only for the uniqueness part.

\begin{lemma} \label{lemma:S} 
    Let $(X,| \cdot |)$ and $Y$ be Banach spaces.  
    Let $f: O \to X$ be $C^s_b$ smooth (with $O \subset X$ an open neighborhood of $0$), $s \in \mathbb{N} \cup \{ \omega \}$, with $f(0) = 0$ and such that $A:= Df(0)$ satisfies $\sigma(A) \subset B_1^{\mathbb{C}}$; denote $\phi = f - A$ and suppose $| \cdot |$ is adapted to $A$ (Remark~\ref{remark:adapt}).
    Let $A_0 \in \mathrm{Aut}(Y)$ such that 
    \begin{equation}
        \Vert A_0^{-1} \Vert_{\mathcal{L}(Y)} \Vert A \Vert_{\mathcal{L}(X)}^{l + 1} < 1 
        \label{eq:ell2}
    \end{equation}
    for some $l \in \mathbb{N}$, $ l \leq s$.
    
    Then, for $\Vert \phi \Vert_{C^s_b(B_1^X;X)}$ sufficiently small, the linear operator 
    \begin{equation}
        \mathcal{S} h = A_0 h - h \circ f
        \label{eq:S}
    \end{equation}
    is an automorphism of $\Gamma_{s,l}(B^{X}_1;Y)$. 
    Moreover, the bound on
    $ \mathcal{S}^{-1} $ is independent of $\Vert \phi \Vert_{C^s_b(B_1^X;X)}$.
\end{lemma}

\begin{proof} 

This is essentially just a restatement of Lemma~3.3 of \cite{Cabre2003a} (see also Theorem~5.1 of \cite{banyaga1996cohomology}), the proof remains virtually the same.  

It is easy to see that $\mathcal{S}$ is bounded.  
We construct an inverse.
Given $\eta \in \Gamma_{s,l}$, 
a solution to $\mathcal{S}h = \eta$ is given by
\begin{equation}
    h = \sum_{j = 0}^\infty (A_0)^{-(j+1)} \eta \circ f^j.
    \label{eq:S_solution}
\end{equation}
This is well defined provided we choose $\Vert \phi \Vert_{C^s_b(B_1^X;X)}$ small enough to ensure $f(B_1^X) \subset B_1^X$ (using Remark~\ref{remark:adapt}).
We
need to show that $h \in \Gamma_{s,l}$, and that such an $h$ is obtained uniquely and boundedly from $\eta$. 
The first assertion amounts to the absolute convergence of \eqref{eq:S_solution} in $\Gamma_{s,l}$;
we end up showing
\begin{equation}
    \sum_{j = 0}^\infty \left\Vert (A_0)^{-(j+1)} \eta \circ f^j \right\Vert_{\Gamma_{s,l}} 
    \leq C \Vert \eta \Vert_{\Gamma_{s,l}}, \qquad \text{for all } \eta \in \Gamma_{s,l},
    \label{eq:S_inv_bound}
\end{equation}
which in addition establishes $\Vert h \Vert_{\Gamma_{s,l}} \leq C \Vert \eta \Vert_{\Gamma_{s,l}}$.

By virtue of $\Gamma_{s,l}$ and Taylor's formula,
we have that 
\begin{equation}
  \Vert D^i \eta(x) \Vert_{\mathrm{M}_i(X;Y)} \leq C \Vert \eta \Vert_{\Gamma_{s,l}} |x|^{l - i + 1}, \qquad x \in B_1^X,
  \label{eq:Di_eta}
\end{equation}
for $0 \leq i \leq l$.  
Moreover, by perhaps further adjusting $\Vert \phi \Vert_{C^s_b(B_1^X;X)}$, it can be ensured that 
\begin{equation}
    \Vert D^i f^j \Vert_{C^0_b} \leq C_i (\Vert A \Vert_{\mathcal{L}(X)} + \varepsilon)^j, \qquad 1 \leq i \leq s, 
    \label{eq:Di_fj}
\end{equation}
where $C_i$ is independent of $j$, with $\varepsilon> 0$ fixed small enough so that 
\begin{equation}
    \Vert A_0^{-1} \Vert_{\mathcal{L}(Y)} (\Vert A \Vert_{\mathcal{L}(X)} + \varepsilon)^{l + 1} < 1.
    \label{eq:hyp3_lemmas}
\end{equation}
Property \eqref{eq:Di_fj} has already been noted in \cite{Cabre2003a} and \cite{banyaga1996cohomology} among others -- the result follows from an argument analogous to the one below.

Via the mean value theorem, \eqref{eq:Di_eta} and \eqref{eq:Di_fj} yield 
\begin{equation} 
    \Vert D^i \eta (f^j(x)) \Vert_{\mathrm{M}_i(X;Y)} \leq C \Vert \eta \Vert_{\Gamma_{s,l}} (\Vert A \Vert_{\mathcal{L}(X)} + \varepsilon)^{j(l - i + 1)} |x|^{l - i + 1}
    \label{eq:Di_eta_fj}
\end{equation}
for $0 \leq i \leq l$.
The $i=0$ case of \eqref{eq:Di_eta_fj} gives the $C^0_b$ part of \eqref{eq:S_inv_bound}.

Bounds on the higher derivatives are obtained via Faà di Bruno's formula,
\begin{displaymath}
    D^k (\eta \circ f^j) = \sum_{i=1}^k \sum_{\substack{m \in \mathbb{N}^i \\ |m| = k}} C_{m} D^i \eta \circ f^j [D^{m_1}f^j,\ldots, D^{m_i}f^j].
\end{displaymath}
Using 
\eqref{eq:Di_fj} and \eqref{eq:Di_eta_fj}, we obtain
\begin{multline} 
    \Vert D^k (\eta \circ f^j) (x) \Vert_{\mathrm{M}_k(X;Y)} \leq C \Vert \eta \Vert_{\Gamma_{s,l}} \bigg(  \sum_{1 \leq i \leq l} (\Vert A \Vert_{\mathcal{L}(X)} + \varepsilon)^{j(l - i + 1)+ji}  |x|^{l - i + 1} \\
     +  \sum_{l < i \leq k} (\Vert A \Vert_{\mathcal{L}(X)} + \varepsilon)^{ji} \bigg).
     \label{eq:Dk_eta_fj}
\end{multline}
For $1 \leq k \leq l$, \eqref{eq:Dk_eta_fj} gives
\begin{displaymath}
    \Vert A_0^{-(j+1)} D^k (\eta \circ f^j) (x) \Vert_{\mathrm{M}_k(X;Y)} \leq C \Vert A_0^{-1} \Vert_{\mathcal{L}(Y)} \Vert \eta \Vert_{\Gamma_{s,l}}   \left( \Vert A_0^{-1} \Vert_{\mathcal{L}(Y)} (\Vert A \Vert_{\mathcal{L}(X)} + \varepsilon)^{l + 1} \right)^j  |x|.
\end{displaymath}
Otherwise, for $l < k \leq s$,
\begin{equation}
    \Vert A_0^{-(j+1)} D^k (\eta \circ f^j) (x) \Vert_{\mathrm{M}_k(X;Y)} \leq C \Vert A_0^{-1} \Vert_{\mathcal{L}(Y)} \Vert \eta \Vert_{\Gamma_{s,l}}   \left( \Vert A_0^{-1} \Vert_{\mathcal{L}(Y)} (\Vert A \Vert_{\mathcal{L}(X)} + \varepsilon)^{l + 1} \right)^j.
    \label{eq:omega_case}
\end{equation}
Via \eqref{eq:hyp3_lemmas}, the last two equations
establish \eqref{eq:S_inv_bound} for all parts of the $\Gamma_{s,l}$ norm and hence the absolute convergence of \eqref{eq:S_solution} 
(in the $s = \omega$ case, \eqref{eq:omega_case} continues to hold and can be used with $k = l +1 $ to conclude the same). 

It remains to address uniqueness.
If $h \in \ker (\mathcal{S})$, then $h = A_0^{-1} h \circ f$, and hence $h = A_0^{-j} h \circ f^j$ for all $j \geq 1$.
But $\Vert A_0^{-j} h \circ f^j \Vert_{\Gamma_{s,l}} \to 0 $ as $j \to \infty$, by \eqref{eq:S_inv_bound} applied with $\eta = h$.
We conclude that $\ker (\mathcal{S}) = \{0\}$.
\end{proof}

We turn to study the nonlinear part of the invariance equation.
Note that if $\pi^> \in \Gamma_{r,\ell}$, the right hand side of \eqref{eq:maineq_split} belongs to $\Gamma_{r,\ell}$ by Lemma~\ref{lemma:approx}.
Hypotheses \ref{hyp1:X}-\ref{hyp3:ell} and
Remark~\ref{remark:adapt} imply that Lemma~\ref{lemma:S} is applicable to the pair $(f,A_0)$ defined in \ref{hyp2:f} with $Y = X_0$ and $(s,l) = (r,\ell)$. 
Hence, we may seek a solution to \eqref{eq:maineq_split} equivalently as a fixed point of 
\begin{equation}
    \mathcal{T}(\pi^>) := \mathcal{S}^{-1} \left(- \psi \circ (\pi^{\leq} + \pi^>) + \pi^{\leq} \circ f - A_0 \pi^{\leq} \right).
    \label{eq:mathcal_T}
\end{equation}

\begin{lemma} \label{lemma:Tfrechet}
Suppose \ref{hyp1:X}-\ref{hyp4:spec1} hold with $r \in \mathbb{N} \cup \{ \omega\}$, that $\Vert \phi \Vert_{C^r_b(B_1^X;X)}$ is as small as in Lemma~\ref{lemma:S} applied to the pair $(f,A_0)$ from \ref{hyp2:f} (with $(s,l) = (r,\ell)$), and that $\psi = \sum_{n = 2}^\ell g_n$ and $\pi^{\leq}$ are the results of Lemma~\ref{lemma:approx}.
Then, the map $\mathcal{T}:\Gamma_{r,\ell} \to \Gamma_{r,\ell}$ from \eqref{eq:mathcal_T}   is well defined and differentiable in the Fréchet sense at all $\pi^> \in \overline{B}_1^{\Gamma_{r,\ell}}$, the closed unit ball in $\Gamma_{r,\ell}$.  
Its derivative is given by  
\begin{equation}
    D_{\pi^>}\mathcal{T}(\pi^>)[\eta] = -\mathcal{S}^{-1}D_x \psi \circ (\pi^{\leq}+\pi^>)[\eta], \qquad \eta  \in \Gamma_{r,\ell}.
    \label{eq:DmathcalT}
\end{equation}
Moreover, 
\begin{equation}
    \left\Vert D\mathcal{T}(\pi^>) \right\Vert_{\mathcal{L}(\Gamma_{r,\ell})} \leq C \left\Vert \mathcal{S}^{-1} \right\Vert_{\mathcal{L}(\Gamma_{r,\ell})} \Vert \psi \Vert_{C^{\ell}_b(B_\rho^{X_0};X_0)}.
    \label{eq:DmathcalTbound}
\end{equation}
for some $\rho>0$ large enough so that $(\pi^{\leq}+\pi^>)(B_1^X) \subset B_{\rho/2}^{X_0}$. \end{lemma}

\begin{proof}
    The assumptions ensure that $\mathcal{T}$ is well defined.
    Differentiability follows in essence from Theorem~15 of \cite{Irwin72}; we will be thorough nonetheless.

    Suppose at first that $r \in \mathbb{N}$.
    We have that $\pi^{\leq} \in C^\infty_b(B_1^{X};X_0)$, and  
    that
    $\psi$ is $C^\infty_b(B_\rho^{X_0};X_0)$, with $\rho>0$ chosen
    such that $(\pi^{\leq}+\pi^>)(B_1^X) \subset B_{\rho/2}^{X_0}$ for all $\pi^> \in \overline{B}_1^{\Gamma_{r,\ell}}$.
    Take $\eta \in \Gamma_{r,\ell}$ sufficiently small so that $(\pi^{\leq}+\pi^>+\eta)(B_1^X) \subset B_{\rho}^{X_0}$, then
    \begin{align}
        \big\Vert \mathcal{T}(\pi^>+\eta)-&\mathcal{T}(\pi^>)  + \mathcal{S}^{-1}D  \psi \circ (\pi^{\leq}+\pi^>)[\eta] \big\Vert_{\Gamma_{r,\ell}} \nonumber  \\
        &= \left\Vert - \int_0^1 \mathcal{S}^{-1}D \psi \circ (\pi^{\leq}+\pi^> +s \eta)[\eta] ds  + \mathcal{S}^{-1}D \psi \circ (\pi^{\leq}+\pi^>)[\eta] \right\Vert_{\Gamma_{r,\ell}} \nonumber \\
        &\leq \left\Vert \mathcal{S}^{-1} \right\Vert_{\mathcal{L}(\Gamma_{r,\ell})}  \int_{[0,1]^2} \left\Vert   
        D^2 \psi \circ (\pi^{\leq}+\pi^> + ts \eta)[\eta,s\eta] \right\Vert_{\Gamma_{r,\ell}} dtds \nonumber \\
        &\leq C\left\Vert \mathcal{S}^{-1} \right\Vert_{\mathcal{L}(\Gamma_{r,\ell})} \Vert \psi \Vert_{C^{r+2}_b(B_\rho^{X_0};X_0)} \left\Vert \eta \right\Vert_{\Gamma_{r,\ell}}^2. \label{eq:DTlemma_proof}
    \end{align}
    Here, the last line used  
    \begin{multline}
        D^k \left( D^2 \psi \circ (\pi + ts \eta)[\eta,s\eta] \right)=
        \sum_{m=0}^k C_m D^2 \psi \circ (\pi + ts \eta) [D^m \eta, sD^{k-m} \eta]  \\
        +\sum_{\substack{n \in \mathbb{N}^2_0 \\ |n| < k}} \sum_{i=1}^{k - |n|} \sum_{\substack{j \in \mathbb{N}^i \\ |j| = k - |n|}} C_{j,n} D^{2+i} \psi \circ (\pi + ts \eta) [D^{n_1} \eta, sD^{n_2} \eta, D^{j_1}(\pi + ts \eta),\ldots,D^{j_i}(\pi + ts \eta)], \label{eq:D2g_exp}
    \end{multline}
    where $\pi = \pi^{\leq}+\pi^>$. 
    By considering the limit $\left\Vert \eta \right\Vert_{\Gamma_{r,\ell}} \to 0$ in \eqref{eq:DTlemma_proof}, we observe that the Fréchet derivative of $ \mathcal{T}$ exists and is given by \eqref{eq:DmathcalT}. 
    The bound \eqref{eq:DmathcalTbound} is obtained by proceeding as in \eqref{eq:D2g_exp},
    and noting that $\Vert \psi \Vert_{C^{k}_b(B_\rho^{X_0};X_0)} = \Vert \psi \Vert_{C^{\ell}_b(B_\rho^{X_0};X_0)}$ for all $k \geq \ell$.
    
    The above line of reasoning is applicable in the $r = \omega$ case as well, using \eqref{eq:D2g_exp} with $k = \ell + 1$.  
\end{proof}

The above lemma exploited the specific, polynomial form of $\psi$.
The situation is a little more dire in the general case, which will be required for the uniqueness portion of Theorem~\ref{thm:maps}.
Indeed, if instead, $\psi$ is an arbitrary $C^{s+1}_b$ function, $\ell \leq s \leq r-1$ (in the $r \in \mathbb{N}$ case), for which $\mathcal{T}$ is well defined,  
the map $\mathcal{T}$ can not be expected to be differentiable $\Gamma_{s+1,\ell} \to \Gamma_{s+1,\ell}$, as the formula for the norm of $D\mathcal{T}$  
would involve $D^{s+2} \psi$. 
Even if we consider $\mathcal{T}:\Gamma_{s,\ell} \to \Gamma_{s,\ell}$ 
it is generally \textit{not} differentiable in the Fréchet sense (unless $\psi$ happens to have uniformly continuous derivatives or $X_0$ is finite dimensional \cite{Irwin72}).
Nonetheless, the finite increments formula continues to hold and is sufficient for our purposes here, as detailed in the following remark.

\begin{remark} \label{remark:DT_replacement}
    Suppose instead, that $\psi \in C^{s+1}_b(B_\rho^{X_0};X_0)$ and $\pi^{\leq} \in C^{s+1}_b(B_1^{X};X_0)$, $\ell \leq s \leq r$ ($r \in \mathbb{N}$), are arbitrary maps (they are no longer necessarily polynomials) with $\pi^{\leq}(0) = 0$ and $D\pi^{\leq} (0) = \widetilde{P}_0$ and $\rho > 0$ chosen such that $(\pi^\leq + \pi^>)(B_1^X) \subset B_{\rho/2}^{X_0}$ for all $\pi^> \in \overline{B}_1^{\Gamma_{s,\ell}}$.  
    If they satisfy the relationship 
    \begin{equation}
        (A_0 + \psi) \circ \pi^{\leq}(x) = \pi^{\leq} \circ f(x) + o(|x|^\ell)
        \label{eq:orderL_replacement}
    \end{equation}
    as $x \to 0$, then, by their assumed smoothness, the right hand side of \eqref{eq:maineq_split},
    \begin{displaymath}
        -\psi \circ (\pi^{\leq} + \pi^>) + \pi^\leq \circ f - A_0 \pi^\leq, 
    \end{displaymath}
    belongs to $\Gamma_{s,\ell}(B_1^X;X_0)$ for the prescribed range of $s$ (at $s = \ell $ the last part of \eqref{eq:Gammasldef} follows from the fact that $\psi,\pi^{\leq}$ and $f$ are at least $C^{s+1}_b$ and \eqref{eq:lastpartgamma}; if $s = r$, then $f$ is merely $C^s_b$, but, since $r > \ell$, \eqref{eq:Gammasldef} does not include the last term).
    Hence, $\mathcal{T}:\Gamma_{s,\ell} \to \Gamma_{s,\ell}$ as in \eqref{eq:mathcal_T} is well defined.
    Using the finite increments formula,
    \begin{displaymath}
        \mathcal{T}(\pi^> + \eta) - \mathcal{T}(\pi^>) = \int_0^1 \frac{d}{dt} \mathcal{T}(\pi^> + t\eta) dt,
    \end{displaymath}
    we may obtain a bound similar to \eqref{eq:DmathcalTbound} with the same procedure,
    \begin{equation}
        \left\Vert \frac{d}{dt} \mathcal{T}(\pi^> + t\eta) \right\Vert_{\Gamma_{s,\ell}} \leq C \left\Vert \mathcal{S}^{-1} \right\Vert_{\mathcal{L}(\Gamma_{s,\ell})} \Vert \psi \Vert_{C^{s+1}_b(B_\rho^{X_0};X_0)} \Vert \eta \Vert_{\Gamma_{s,\ell}},
        \label{eq:contraction_replacement}
    \end{equation}
    for $\eta \in \Gamma_{s,\ell}$ small enough and $\pi^> \in \overline{B}_1^{\Gamma_{s,\ell}}$,  
    which will be sufficient to establish the contraction property in the lemma below and eventually the uniqueness claim of Theorem~\ref{thm:maps}.
    The preceding argument is also applicable in the $r = \omega$ case if both $\pi^\leq$ and $\psi$ are $C^\omega_b$ maps. 
    In this case, both '$s$' and '$s+1$' in \eqref{eq:contraction_replacement} should be interpreted as '$\omega$', which is a convention we shall stick to in what follows.
\end{remark}

\begin{lemma} \label{lemma:contraction}
    Assume \ref{hyp1:X}-\ref{hyp4:spec1} with $r \in \mathbb{N} \cup \{\omega \}$, that $\Vert \phi \Vert_{C^r_b(B_1^X;X)}$ is sufficiently small, and that $\psi = \sum_{n = 2}^\ell g_n$ and $\pi^{\leq}$ are the results of Lemma~\ref{lemma:approx}.
    Then, with $\mathcal{T}$ as in \eqref{eq:mathcal_T},
    \begin{equation}
        \mathcal{T}(\pi^>) = \pi^>
        \label{eq:FPeq}
    \end{equation}
    has a unique solution in $\overline{B}_1^{\Gamma_{r,\ell}}$.

    If instead, $\psi \in C^{s+1}_b(B_\rho^{X_0};X_0)$\footnote{Here, $\rho > 0$ is such that $(\pi^\leq + \pi^>)(B_1^X) \subset B_{\rho/2}^{X_0}$ for all $\pi^> \in \overline{B}_1^{\Gamma_{s,\ell}}$.} and $\pi^{\leq} \in C^{s+1}_b(B_1^{X};X_0)$, $\ell \leq s \leq r$ (if $r=\omega$, $\psi$ and $\pi^{\leq}$ should be $C^\omega_b$), are arbitrary maps satisfying \eqref{eq:orderL_replacement} with $\pi^{\leq}(0) = 0$ and $D\pi^{\leq} (0) = \widetilde{P}_0$;
    and moreover, $\Vert \pi^{\leq} - \widetilde{P}_0 \Vert_{C^{s+1}_b(B_1^{X};X_0)}$ and $\Vert \psi \Vert_{C^{s+1}_b(B_\rho^{X_0};X_0)}$ are sufficiently small, then \eqref{eq:FPeq} has a unique solution in $\overline{B}_1^{\Gamma_{s,\ell}}$.
\end{lemma}

\begin{proof}
    \textit{Uniqueness.} 
    That $\mathcal{T}$ is a contraction in $\overline{B}_1^{\Gamma_{r,\ell}}$ follows from \eqref{eq:DmathcalTbound} and \eqref{eq:gi_pi_bound}. 
    First, we note that there exists $\rho > 0$ such that $(\pi^{\leq} + \pi^>)(B_1^X) \subset B_{\rho/2}^{X_0}$ for all $\pi^> \in \overline{B}_1^{\Gamma_{r,\ell}}$.
    Then, by means of \eqref{eq:gi_bound}, we may ensure that $\Vert \psi \Vert_{C^{\ell}_b(B_\rho^{X_0};X_0)}$ is sufficiently small (via adjusting $\Vert \phi \Vert_{C^r_b(B_1^X;X)}$) so that $\left\Vert D\mathcal{T}(\pi^>) \right\Vert_{\mathcal{L}(\Gamma_{r,\ell})} < 1$ for all $\pi^> \in \overline{B}_1^{\Gamma_{r,\ell}}$ (using Lemma~\ref{lemma:S}; in particular, that $\left\Vert \mathcal{S}^{-1} \right\Vert_{\mathcal{L}(\Gamma_{r,\ell})}$ is independent of further shrinking of $\Vert \phi \Vert_{C^r_b(B_1^X;X)}$ beyond the minimum threshold).

    \textit{Existence.}
    We have that
    \begin{displaymath}
        \left\Vert \mathcal{T}(\pi^>) \right\Vert_{\Gamma_{r,\ell}} \leq \left\Vert \mathcal{T}(0) \right\Vert_{\Gamma_{r,\ell}} + \sup_{y \in [0:\pi^>]} \left\Vert D\mathcal{T}(y) \right\Vert_{\mathcal{L}(\Gamma_{r,\ell})} \left\Vert \pi^> \right\Vert_{\Gamma_{r,\ell}}.
    \end{displaymath}
    If $\pi^> \in \overline{B}_1^{\Gamma_{r,\ell}}$, the last term is strictly less than one by the preceding paragraph.
    The $\Gamma_{r,\ell}$ norm of $\mathcal{T}(0) = \mathcal{S}^{-1}(-g \circ \pi^{\leq} + \pi^{\leq} \circ f)$ is adjustable freely by $\Vert \phi \Vert_{C^r_b(B_1^X;X)}$ (using Lemmas \ref{lemma:S} and \ref{lemma:approx}).
    Thus $\mathcal{T}(\overline{B}_1^{\Gamma_{r,\ell}}) \subset \overline{B}_1^{\Gamma_{r,\ell}}$ can be guaranteed for $\Vert \phi \Vert_{C^r_b(B_1^X;X)}$ sufficiently small.

    The last claim follows upon replacing \eqref{eq:DmathcalTbound} by \eqref{eq:contraction_replacement} in the preceding argument, via the route explained in Remark~\ref{remark:DT_replacement}.  
    The smallness of $\Vert \pi^{\leq} - \widetilde{P}_0 \Vert_{C^{s+1}_b(B_1^{X};X_0)}$ is used to ensure $\mathcal{T}(0)$ is freely adjustable in the $\Gamma_{s,\ell}$ norm.  
\end{proof}

\begin{remark} \label{remark:half_unit_ball}
    Up to shrinking $\Vert \phi \Vert_{C^r_b(B_1^X;X)}$ further (along with $\Vert \pi^{\leq} - \widetilde{P}_0 \Vert_{C^{s+1}_b(B_1^{X};X_0)}$ and $\Vert \psi \Vert_{C^{s+1}_b(B_\rho^{X_0};X_0)}$ in the second case), we may as well set up the contraction mapping on $\overline{B}_\gamma^{\Gamma_{s,\ell}}$ for $0 < \gamma < 1$, since both $\left\Vert D\mathcal{T}(\pi^>) \right\Vert_{\mathcal{L}(\Gamma_{r,\ell})}$ and $\mathcal{T}(0)$ are freely adjustable via these quantities.
\end{remark}

\begin{proof}[Proof of Theorem~\ref{thm:maps}\ref{thm:maps_1}]
For $\delta > 0$, and a map $h$, consider its scaled version
$h_\delta : x \mapsto (1/\delta) h (\delta x)$.
We wish to use $f \mapsto f_\delta$.
A transformation as such leaves the linear part invariant, as $Df_\delta (0) = Df(0) = A$.
In particular, hypotheses \ref{hyp3:ell}-\ref{hyp5:spec2} remain unchanged.
Moreover, 
\begin{displaymath}
    \Vert \phi_\delta \Vert_{C^r_b(B_1^X;X)} = \Vert f_\delta - A \Vert_{C^r_b(B_1^X;X)} = o(1)
\end{displaymath}
as $\delta \to 0$.
Adjusting $\delta$ and hence $\Vert \phi_\delta \Vert_{C^r_b(B_1^X;X)}$ as necessary,
we may apply Lemmas \ref{lemma:approx}-\ref{lemma:contraction} to $f_\delta$ and obtain $\hat{g}$ and $\hat{\pi}$ such that
\begin{displaymath}
    \hat{g} \circ \hat{\pi} = \hat{\pi} \circ f_\delta
\end{displaymath}
holds on $B_1^X$.
Now, upon defining $\pi(x) : =\delta \hat{\pi} (x / \delta) $ and $g(x) := \delta \hat{g} (x / \delta) $, we obtain the desired relationship 
\begin{displaymath}
    g \circ \pi = \pi \circ f
\end{displaymath}
on $B_\delta^X$.
\end{proof}

\begin{remark} \label{remark:scaling} 
The above scaling argument will be used twice more in the following form.
Given $C^k_b$ maps $f: X \supset U \to U$, $\pi^{\leq}:  U \to X_0$ and $g: X_0 \supset O_0 \to X_0$,  
suppose we have obtained a solution $\hat{\pi}^>$ such that
\begin{displaymath}
    g_\delta \circ (\pi^{\leq}_\delta + \hat{\pi}^>) =  (\pi^{\leq}_\delta + \hat{\pi}^>) \circ f_\delta
\end{displaymath}
holds on $B_1^X$,  
for some $\delta > 0$.
Then, if $\pi^> (x): = \delta \hat{\pi}^> (x/\delta)$,
\begin{displaymath}
    g \circ (\pi^\leq + \pi^>) = (\pi^\leq + \pi^>) \circ f
\end{displaymath}
holds on $B_\delta^X$.
\end{remark}

The following lemma is the first step towards establishing the uniqueness claim \ref{thm:maps}\ref{thm:maps_2}.

\begin{lemma} \label{lemma:uniq1} 
    Suppose hypotheses \ref{hyp1:X}-\ref{hyp4:spec1} hold, and let $g$ be a $C^{s+1}_b(O_0;X_0)$ map, $ \ell  \leq s \leq r$ ($r \in \mathbb{N} \cup \{ \omega \}$; if $r = \omega$, $g$ should be $C^\omega_b$), on some neighborhood of $0$, $O_0 \subset X_0$, with $g(0) = 0$ and $Dg(0) = A_0$. 
    Assume moreover either of the following:
    \begin{enumerate}
        \item \ref{hyp5:spec2} holds with $m=2$;
        \item \label{g_uniq_b} We are given a degree-$\ell$ polynomial on $X_0$, $\pi^0 = \sum_{n=2}^\ell \pi^0_n$, such that $(g_n,\pi_{(0,\ldots,0)}) = (D^ng(0),\pi^0_n)$ solves \eqref{eq:order_n_eq1} for all $2 \leq n \leq \ell$.
    \end{enumerate}
    Then, there exists a unique $C^{s}_b$ map $\pi$ with $D\pi(0) \vert_{X_0} = \mathrm{id}_{X_0}$ such that $(g,\pi)$ satisfies \eqref{eq:map_foliation_diag} (with $U$ sufficiently small).

    In case (\ref{g_uniq_b}), uniqueness is only guaranteed within the subclass of functions 
    satisfying $D^n\pi(0) \vert_{X_0} = \pi^0_n$ for $2 \leq n \leq \ell$.
\end{lemma}

\begin{proof}

    In either setting, the solutions to \eqref{eq:order_n_eq1} are uniquely determined upon inserting $g_n = D^ng(0)$.
    Hence, Lemma~\ref{lemma:approx} gives a unique degree-$\ell$ polynomial $\pi^{\leq}$, with $\pi^{\leq}(0) = 0$ and $D\pi^\leq (0) = \widetilde{P}_0$, such that \eqref{eq:approx_lemma} holds (see Remark~\ref{remark:uniqueness_approxlemma}). 
    
    Once more, via the scaling argument it can be ensured that Lemma~\ref{lemma:S} is applicable on $\Gamma_{s,\ell}$.
    The same scaling argument can be used to adjust the $C^{s+1}_b$ norms of $\psi = g - A_0$ and $\pi^\leq-\widetilde{P}_0$ (see Remark~\ref{remark:scaling}) to match the assumptions of Lemma~\ref{lemma:contraction}. 
    In particular, 
    \begin{displaymath}
        \sup_{x \in B_1^X} | \pi^\leq_\delta(x) | \leq  \Vert D \pi^\leq (0) \Vert_{\mathcal{L}(X;X_0)} + o(1)
    \end{displaymath}
    as the scaling parameter $\delta$ tends to zero, and hence there exists a finite $\rho > 0$ for which $(\pi^\leq_\delta + \hat{\pi}^>)(B_1^X) \subset B_{\rho/2}^{X_0}$ for all $\hat{\pi}^> \in \overline{B}_1^{\Gamma_{s,\ell}}$.
    Up to further shrinking $\delta$, we may assume $\psi_\delta$ is defined and sufficiently small in the $C^{s+1}_b$ norm on $ B_{\rho}^{X_0}$, and similarly for $\Vert \pi^{\leq}_\delta - \widetilde{P}_0 \Vert_{C^{s+1}_b(B_1^{X};X_0)}$.
    Then, via Lemma~\ref{lemma:contraction} (and Remark~\ref{remark:DT_replacement}) we obtain a unique $\hat{\pi}^> \in \overline{B}^{\Gamma_{s,\ell}}_1$ such that
    \begin{displaymath}
        A_0 \hat{\pi}^> - \hat{\pi}^> \circ f = - \psi_\delta \circ (\pi^{\leq}_\delta + \hat{\pi}^>) + \pi^{\leq}_\delta \circ f_\delta - A_0 \pi^{\leq}_\delta
    \end{displaymath}
    holds on $B_1^X$.
    Remark~\ref{remark:scaling} now finishes the proof. 
\end{proof}

A particular corollary of the above lemma (and Remark~\ref{remark:linearpart}) is that linearizing semiconjugacies are uniquely determined by $D \pi (0)$ (cf.\ Theorem~1 of \cite{kvalheim2021existence}).

\begin{proof}[Proof of the $r = \infty$ case of Theorem~\ref{thm:maps}\ref{thm:maps_1}]
By assumption, we have $f \in C^\infty_b$.  
Let $g$ and $\pi^\leq$ be fixed as the polynomials obtained in Lemma~\ref{lemma:approx}, and let $\pi = \pi^\leq + \pi^>$ denote the full $C^{\ell+1}_b$ solution.

We may find $\rho > 0$ such that $g$ is a $C^\infty_b$ diffeomorphism on $B_{\rho}^{X_0}$.
According to the above procedure, we may construct a $C^r_b$ solution  $\pi_r = \pi^\leq +\pi_r^>$ for each $r \geq \ell +1 $.
This solution satisfies \eqref{eq:map_foliation_diag} on $B_{\delta_r}^X$, for some $\delta_r > 0$ -- we may assume $\delta_r$ has been chosen small enough so that $B_{\delta_r}^X \subset U$ and hence $\pi \vert_{B_{\delta_r}^X} = \pi_r$, where $U$ is the neighborhood on which uniqueness is satisfied from Lemma~\ref{lemma:uniq1} applied to the $C^{\ell+1}_b$ case.
Take $\gamma > 0$ small enough such that $B^{X}_{\gamma} \subset \pi^{-1} ( B_\rho^{X_0} )$, and choose $j_r \in \mathbb{N}$ large enough to satisfy
$f^{j_r}( B^{X}_{\gamma} ) \subset B_{\delta_r}^X$.
Then,
\begin{equation}
    \hat{\pi}_r : = g^{-j_r} \circ \pi_r \circ f^{j_r}
    \label{eq:extension}
\end{equation}
is a $C^r_b$ extension of $\pi_r$ to $B^{X}_{\gamma}$.
Note that \eqref{eq:extension} is well defined since $\pi_r \circ f^{j_r} ( B^{X}_{\gamma} ) = \pi \circ f^{j_r} ( B^{X}_{\gamma}  ) \subset g^{j_r} (B_\rho^{X_0})$, using the fact that $\pi_r = \pi$ on $B_{\delta_r}^X$, that $f(B^{X}_{\gamma} ) \subset B^{X}_{\gamma} $, and that $\pi \circ f = g \circ \pi$ holds on $B^{X}_{\gamma}$.
Applying the uniqueness claim once more to \eqref{eq:extension} (noting that the $\ell$-jet of $\hat{\pi}_r$ agrees with $\pi^\leq$), we find a uniform neighborhood $B^{X}_{\gamma} \cap U$ on which $\pi = \hat{\pi}_r $ for all $r$. 
We have thus shown that $\pi$ is $C^r_b$ on $B^{X}_{\gamma} \cap U$ for all $r \in \mathbb{N}$, i.e., $\pi \in C^\infty_b$. 
\end{proof}

\begin{proof}[Proof of Theorem~\ref{thm:maps}\ref{thm:maps_2}]
    Suppose we are given two $C^{k}_b$, $\ell +1 \leq k \leq r$, pairs of solutions to \eqref{eq:map_foliation_diag}, $(g,\pi)$ and $(\tilde{g},\tilde{\pi})$, as in the statement, and assume for now that $r \in \mathbb{N} \cup \{ \infty \}$.
    We may assume $D\pi(0) \vert_{X_0} = D\tilde{\pi}(0) \vert_{X_0} = \mathrm{id}_{X_0}$ (and hence also $Dg(0) = D\tilde{g}(0)= A_0$) by considering 
    \begin{equation}
        \left((D\pi(0) \vert_{X_0})^{-1} g \circ D\pi(0) \vert_{X_0},( D\pi(0) \vert_{X_0})^{-1} \pi \right)
        \label{eq:lastproof_firststep}
    \end{equation}
    in place of $(g,\pi)$, and similarly for $(\tilde{g},\tilde{\pi})$.
    We may also assume that $\pi$ and $\tilde{\pi}$ are defined on the same ball around the origin, as the assertion is local.

    Since $D\pi(0) \vert_{X_0} = D\tilde{\pi}(0) \vert_{X_0} = \mathrm{id}_{X_0}$, there exists a neighborhood $U_0$ of $0$ in $X_0$ such that both $\pi \vert_{U_0}$ and $\tilde{\pi} \vert_{U_0}$ are $C^{k}_b$ diffeomorphisms onto their images.
    Let $\tilde{\theta} : = \tilde{\pi} \vert_{U_0} \circ ( \pi \vert_{U_0})^{-1}$.
    It is clear that $(\tilde{\theta} \circ g \circ \tilde{\theta}^{-1}, \tilde{\theta} \circ \pi)$ is another solution to \eqref{eq:map_foliation_diag} (Remark~\ref{remark:diffeo}).
    It is a priori unclear whether this agrees with $(\tilde{g},\tilde{\pi})$, but we have by $\tilde{\theta} \circ \pi \vert_{U_0} = \tilde{\pi} \vert_{U_0}$ and Remark~\ref{remark:uniqueness_approxlemma} that 
    \begin{displaymath}
       D^n ( \tilde{\theta} \circ g \circ \tilde{\theta}^{-1}) (0)  = D^n \tilde{g} (0), \qquad 0 \leq n \leq \ell,
    \end{displaymath} 
    (cases $n=0,1$ are by assumption), which in turn implies\footnote{This claim is not entirely obvious, but it follows from a rather uneventful, straightforward computation, which is hence omitted.}
    \begin{equation}
        D^n ( \tilde{\theta} \circ g ) (0)  = D^n ( \tilde{g} \circ \tilde{\theta} ) (0), \qquad 0 \leq n \leq \ell.
        \label{eq:g_jet}
    \end{equation}

    We seek the unknown diffeomorphism in \eqref{eq:uniqueness} as $\theta = \tilde{\theta} + h $ for some higher order function $h$ on $X_0$ to be determined. 
    In particular, $h$ should satisfy
    \begin{equation} 
         \tilde{g} \circ (\tilde{\theta} +  h)  =  (\tilde{\theta} + h ) \circ g .
       \label{eq:uniq_h}
    \end{equation}
    We shall treat this problem similarly to \eqref{eq:map_foliation_diag}.
    Let
    \begin{displaymath}
        {\mathcal{S}} \hat{h} = A_0 \hat{h} - \hat{h} \circ  g_\delta 
    \end{displaymath}
    denote the linear part of the scaled version of \eqref{eq:uniq_h} (analogously to \eqref{eq:S}). 
    Note that $Dg_\delta(0) = A_0$ and \eqref{eq:ell2} holds (trivially) with $A_0$ in place of $A$. 
    We may choose $\delta$ such that that $g_\delta-A_0$ is sufficiently small in the $C^{\ell+1}_b(B_1^{X_0};X_0)$ norm, and hence 
    the assumptions of Lemma~\ref{lemma:S} are  satisfied (with $X,Y = X_0$).
    Therefore, ${\mathcal{S}}$ is an automorphism of $\Gamma_{\ell,\ell}(B_1^{X_0};X_0)$, where $\ell$ is fixed as in \ref{hyp3:ell} (for later use, so that $h$ will vanish to sufficiently high order). 
    We repeat the scaling argument from the proof of Lemma~\ref{lemma:uniq1}.
    In particular,
    $\Vert \tilde{\theta}_\delta \Vert_{C^0_b(B_1^{X_0};X_0)} \leq 1 + o(1) $ as the scaling parameter $\delta$ tends to zero, and hence, we may choose some $\rho > 4$ and adjust $\delta$ so that $(\tilde{\theta}_\delta + \hat{h}) (B_1^{X_0}) \subset B_{\rho/2}^{X_0}$ for all $\hat{h} \in \overline{B}_{1/4}^{\Gamma_{\ell,\ell}}$, then, up to perhaps further shrinking $\delta$, we may assume $\tilde{\psi}_\delta = \tilde{g}_\delta - A_0$ is sufficiently small on $B_\rho^{X_0}$; $\tilde{\theta}_\delta - \mathrm{id}_{X_0}$ and $g_\delta-A_0$ are sufficiently small on $B_1^{X_0}$ in the $C^{\ell+1}_b$ norm, to the degree that allows us to set up the contraction mapping on $\overline{B}_{1/4}^{\Gamma_{\ell,\ell}}$ (see Remark~\ref{remark:half_unit_ball}). 
    By \eqref{eq:g_jet}, we have that the pair $(\tilde{\psi}_\delta,\tilde{\theta}_\delta)$ satisfies
    \begin{displaymath}
        (A_0 + \tilde{\psi}_\delta ) \circ \tilde{\theta}_\delta (x) = \tilde{\theta}_\delta \circ g_\delta (x) +o(|x|^\ell)
    \end{displaymath}
    as $x \to 0$
    (in place of \eqref{eq:orderL_replacement}), and hence the assumptions of Lemma~\ref{lemma:contraction} are met with $X = X_0$ and $s = \ell$.
    We are thus provided with a unique fixed point $\hat{h} \in \overline{B}_{1/4}^{\Gamma_{\ell,\ell}(B_{1}^{X_0};X_0)}$ of the 
    contraction mapping
    \begin{displaymath}
        {\mathcal{T}}(\hat{h})  :=  {\mathcal{S}}^{-1} ( -\tilde{\psi}_\delta \circ (\tilde{\theta}_\delta +  \hat{h}) + \tilde{\theta}_\delta \circ g_\delta  - A_0 \tilde{\theta}_\delta ).
    \end{displaymath}

    It can always be ensured that $\tilde{\theta}_\delta +  \hat{h}$ is a $C^{\ell}_b$ diffeomorphism from $B_{1/2}^{X_0}$ to its image, up to adjustments in $\delta$ depending solely on $\Vert \tilde{\theta} - \mathrm{id}_{X_0} \Vert_{C^{\ell+1}_b}$.
    Indeed, considering $ \chi(y_0,x_0)  : = y_0 + x_0 - \tilde{\theta}_\delta(x_0) -  \hat{h}(x_0)$, we can guarantee that
    \begin{displaymath}
        \sup_{y_0 \in \overline{B}_{1/4}^{X_0}, x_0 \in \overline{B}_{1/2}^{X_0}} \Vert D_{x_0} \chi(y_0,x_0) \Vert_{\mathcal{L}(X_0)} \leq \Vert \tilde{\theta}_\delta - \mathrm{id}_{X_0} \Vert_{C^{\ell+1}_b(B_1^{X_0};X_0)} + \Vert  \hat{h} \Vert_{\Gamma_{\ell,\ell}(B_1^{X_0};X_0)} < \frac12,
    \end{displaymath}
    by choosing $\delta$ such that $\Vert \tilde{\theta}_\delta - \mathrm{id}_{X_0} \Vert_{C^{\ell+1}_b(B_1^{X_0};X_0)} < 1/4$.
    Hence, $\chi : \overline{B}_{1/4}^{X_0} \times \overline{B}_{1/2}^{X_0} \to \overline{B}_{1/2}^{X_0}$ 
    is a uniform contraction on the second factor.
    By Theorem~21 of \cite{Irwin72}, $\tilde{\theta}_\delta +  \hat{h}$ 
    is a $C^{\ell}_b$ diffeomorphism on $B^{X_0}_{1/2} \cap (\tilde{\theta}_\delta +  \hat{h})^{-1} (B^{X_0}_{1/4}) \supset B^{X_0}_{1/6}$. 
    
    Via Remark~\ref{remark:scaling}, we have thus obtained $\theta = \tilde{\theta}+h $, a $C^{\ell}_b$ diffeomorphism on $B_{\delta/6}^{X_0}$,
    such that
    \begin{equation}
        \tilde{g} = \theta \circ g \circ \theta^{-1}
        \label{eq:g_conj}
    \end{equation}
    holds on $\theta (B_{\delta/6}^{X_0})$. 
    We emphasize once more that $\delta$, and hence the neighborhood on which \eqref{eq:g_conj} holds, was determined by the original sizes of $\Vert g - A_0\Vert_{C^{\ell+1}_b}$, $\Vert \tilde{g} - A_0 \Vert_{C^{\ell+1}_b}$ and $\Vert \tilde{\theta} - \mathrm{id}_{X_0} \Vert_{C^{\ell+1}_b}$.

    Now $(\theta \circ g \circ \theta^{-1},\theta \circ \pi)$ is yet another solution of \eqref{eq:map_foliation_diag}, with $\theta \circ g \circ \theta^{-1} =  \tilde{g} \in C^{k}_b$ and $\theta \circ \pi \in C^{\ell}_b$.
    It remains to check that it agrees with $(\tilde{g},\tilde{\pi})$.
    By Lemma~\ref{lemma:uniq1} and \eqref{eq:g_conj} it suffices to check that
    \begin{displaymath}
        D^n \tilde{\pi} (0) \vert_{X_0} = D^n (\theta \circ \pi) (0) \vert_{X_0}, \qquad 2 \leq n \leq \ell;
    \end{displaymath}
    but this is obvious, since
    \begin{align*}
        D^n (\theta \circ \pi) (0) \vert_{X_0} &= D^n (\tilde{\theta} \circ \pi) (0)  \vert_{X_0} + D^n (h \circ \pi) (0)\vert_{X_0} \\
        &= D^n (\tilde{\theta} \circ \pi \vert_{X_0}) (0)  \\
        &= D^n \tilde{\pi} (0) \vert_{X_0}.
    \end{align*}
    for $n \leq \ell$.

    If $U \subset X$ is the neighborhood from Lemma~\ref{lemma:uniq1} on which uniqueness of $\tilde{\pi}$ holds (depending only on  $\Vert \tilde{g} - A_0 \Vert_{C^{\ell+1}_b}$ and $\Vert f- A \Vert_{C^{\ell+1}_b}$), then $\tilde{\pi} = \theta \circ \pi$ holds on $\widetilde{U} = U \cap \pi^{-1}(B_{\delta/6}^{X_0})$. 
    It follows, with $\widetilde{U}_0 =\widetilde{U} \cap U_0$, that 
    \begin{displaymath}
        \theta \vert_{\pi(\widetilde{U}_0)}= \tilde{\pi} \vert_{\widetilde{U}_0} \circ ( \pi \vert_{\widetilde{U}_0})^{-1},
    \end{displaymath}
    (i.e., in retrospect, $h = 0$) and hence $\theta$ is (locally) uniquely determined and $C^{k}_b$. 
    The uniqueness claim \eqref{eq:uniqueness} follows with $W = \pi^{-1}(\pi(\widetilde{U}_0)) \cap \widetilde{U}$.

    To see the final claim, if $g$ and $\tilde{g}$ are both linear, then, after the first step \eqref{eq:lastproof_firststep}, we may immediately apply Lemma~\ref{lemma:uniq1} (with case (\ref{g_uniq_b})) to conclude $\theta = D\tilde{\pi}(0)\vert_{X_0} (D \pi(0)\vert_{X_0})^{-1} \in \mathrm{Aut}(X_0)$. 

    The same procedure works for the $r = \omega$ case, with the appropriate modifications (e.g., the contraction mapping $\mathcal{T}$ above should be established on the space  $\Gamma_{\omega,\ell}(B_1^{X_0};X_0)$ instead of $\Gamma_{\ell,\ell}(B_1^{X_0};X_0)$).
\end{proof}

\begin{proof}[Proof of Theorem~\ref{thm:sf}] 
The supposed $X_1$-invariance of $D \varphi_t(0)$ implies that $t \mapsto \widetilde{P}_0D\varphi_t(0) \imath_0$ is a semigroup on $X_0$.
Indeed, for $t,s \geq 0$,
\begin{align*}
    \widetilde{P}_0D\varphi_{t+s}(0) \imath_0 &= \widetilde{P}_0D\varphi_{t}(0) D\varphi_{s}(0) \imath_0  \\
    &= \widetilde{P}_0D\varphi_{t}(0) (P_0 + P_1) D\varphi_{s}(0) \imath_0 \\
    &= \widetilde{P}_0D\varphi_{t}(0)  \imath_0 \widetilde{P}_0  D\varphi_{s}(0) \imath_0,
\end{align*}
since $\widetilde{P}_0D\varphi_{t}(0)  P_1 = 0$ for all $t \geq 0$.
The hypothesis $X_0 \subset \mathrm{dom}(G)$ now implies that $t \mapsto \widetilde{P}_0D\varphi_t(0) \imath_0$ has a bounded generator $\widetilde{P}_0 G \imath_0$ such that $\widetilde{P}_0D\varphi_t(0) \imath_0 = e^{t \widetilde{P}_0 G \imath_0}$, and hence it extends to define a group (of operators on $X_0$), with $\widetilde{P}_0D\varphi_t(0) \imath_0 \in \mathrm{Aut}(X_0)$ for each $t \in \mathbb{R}$
(see Lemma 7.1.18, \cite{buhler2018functional}). 
In particular, $0 \notin \sigma(\widetilde{P}_0D\varphi_\tau(0) \imath_0)$.

This latter observation combined with the remainder of the assumptions ensure that we may 
apply Theorem~\ref{thm:maps} to $\varphi_{\tau} \vert_{\mathcal{D}_{\tau}^\varphi \cap O}: \mathcal{D}^\varphi_{\tau} \cap O  \to X$ and hence obtain $W \subset \mathcal{D}^\varphi_{\tau} \cap O$, an open neighborhood of $0$ that is mapped to itself by $\varphi_\tau$, and the commuting diagram of $C^r_b$ maps 
\begin{equation}
\begin{tikzcd}
W \arrow[r, "\varphi_\tau"] \arrow[d, "\pi"']
& W \arrow[d, "\pi"] \\
X_0 \arrow[r, "g"']
&  X_0
\end{tikzcd}
\label{eq:diag_sfproof}
\end{equation}
Here, we may choose $D\pi(0) = \widetilde{P}_0$ (Remark~\ref{remark:linearpart}).

We intend to use the uniqueness claim (part~\ref{thm:maps_2}) of Theorem~\ref{thm:maps} to obtain the semiflow $\vartheta_t$ in \eqref{eq:sf_foliation_diag}.
By Lyapunov stability of $0$, we may choose $\widetilde{W} \subset W$ such that $\varphi_s(\widetilde{W}) \subset W$ for all $s \geq 0$ and $\varphi_\tau (\widetilde{W}) \subset \widetilde{W}$ (using Remark~\ref{remark:adapt} and the scaling argument introduced in the proof of Theorem~\ref{thm:maps}).
Thus, we may precompose \eqref{eq:diag_sfproof} by $\varphi_s \vert_{\widetilde{W}}$ to obtain
\begin{displaymath}
    g \circ \pi \circ \varphi_s = \pi \circ \varphi_\tau \circ \varphi_s 
    = \pi \circ \varphi_s \circ \varphi_\tau,
\end{displaymath}
which holds on $\widetilde{W}$,
and hence 
\begin{equation} \label{eq:sfproof_diag2}
\begin{tikzcd}
\widetilde{W} \arrow[r, "\varphi_\tau"] \arrow[d, "\pi \circ \varphi_s"']
& \widetilde{W} \arrow[d, "\pi \circ \varphi_s"] \\
X_0 \arrow[r, "g"']
&  X_0
\end{tikzcd} 
\end{equation}
is another candidate for a $C^r_b$ solution of \eqref{eq:map_foliation_diag} with $f = \varphi_\tau$.
That $\pi \circ \varphi_s (0) = 0$ is given by assumption, 
and $D (\pi \circ \varphi_s)(0) \vert_{X_0} = \widetilde{P}_0 D \varphi_s (0) \imath_0 \in \mathrm{Aut}(X_0)$ has already been established.
The conditions to apply the uniqueness claim are hence met, but 
we must make sure that the neighborhood on which
the conjugacy \eqref{eq:uniqueness}
holds can be chosen uniformly with respect to $s$ in order to meet the requirements posed by \eqref{eq:sf_foliation_diag}.

Fix some $0 < \varepsilon < 1/(2 \sup_{s \in [0,2\tau]} \Vert e^{-\widetilde{P}_0 G \imath_0 s}\Vert_{\mathcal{L}(X_0)})$. 
Denote by $\eta_s : = \pi \circ \varphi_s \circ \imath_0$.
Choose $\gamma > 0$ such that, for $t,s \in [0,2\tau]$, $| t-s | < \gamma$ implies $\Vert D^i \eta_t(0) - D^i \eta_s(0) \Vert_{\mathrm{M}_i(X_0;X_0)} < \varepsilon/2$ for all $i = 0,\ldots,\ell+1$. 
By the joint continuity part of assumption \ref{hypA2}, for any $t \in [0,2\tau]$, there exist $\gamma_t < \gamma$ and $\delta_t > 0$ such that\footnote{Here, $\tilde{\imath}_0$ is the inclusion $\mathbb{R}^{\geq 0} \times X_0 \xhookrightarrow{} \mathbb{R}^{\geq 0} \times X$.}
\begin{displaymath}
    D^i(\pi \circ \varphi \circ \tilde{\imath}_0) \big( (t-\gamma_t, t + \gamma_t) \times B_{\delta_t}^{X_0} \big) \subset B^{\mathrm{M}_i(X_0;X_0)}_{\varepsilon/2}(D^i\eta_t(0)), \qquad 0 \leq i \leq \ell+1. 
\end{displaymath} 
There exist a finite set of $t_j$, $j = 1, \ldots, n$, such that the collection $I_j = (t_j-\gamma_{t_j},t_j+\gamma_{t_j})$, $j = 1, \ldots, n$, covers $[0,2\tau]$. 
Setting $\delta := \min_{j=1,\ldots,n} \delta_{t_j}$, 
we have
\begin{align*}
    \sup_{(s,x_0) \in [0,2\tau] \times B_\delta^{X_0}  } \Vert D^i \eta_s (x_0) - &D^i \eta_s (0) \Vert_{\mathrm{M}_i(X_0;X_0)}    \\
      &\leq \max_{j=1,\ldots,n} \sup_{(s,x_0) \in I_j \times B_\delta^{X_0}  } \Vert D^i \eta_s (x_0) - D^i \eta_{t_j} (0) \Vert_{\mathrm{M}_i(X_0;X_0)}  \\
      &+ \max_{j=1,\ldots,n} \sup_{s \in I_j} \Vert D^i \eta_{t_j} (0) - D^i \eta_s (0) \Vert_{\mathrm{M}_i(X_0;X_0)} < \varepsilon
\end{align*} 
for all $i = 0,\ldots,\ell+1$. 
Consequently,
\begin{displaymath}
    \sup_{s \in [0,2\tau]} \Vert \eta_s \Vert_{C^{\ell+1}_b(B_\delta^{X_0};X_0)} \leq \sup_{s \in [0,2\tau], 0 \leq i \leq \ell+1} \Vert D^i \eta_s(0) \Vert_{\mathrm{M}_i(X_0;X_0)} + \varepsilon,
\end{displaymath}
i.e., $\pi \circ \varphi_s \vert_{B^{X_0}_\delta}$ can be uniformly bounded in the $C^{\ell+1}_b(B_\delta^{X_0};X_0)$ norm with respect to $s \in [0,2\tau]$.

Restricted to the domain $B^{X_0}_\delta \cap \eta_s^{-1}(B^{X_0}_{\delta/2}) \supset B^{X_0}_{\delta/3}$, we have that $\pi \circ \varphi_s$
is a $C^r_b$ diffeomorphism across all $s \in [0,2\tau]$.
This follows trivially from Theorem~21 of \cite{Irwin72} applied to the map
$ \chi(s,y_0,x_0)  : = y_0 + x_0 - (D \eta_s (0))^{-1} \eta_s (x_0)$, $[0,2\tau] \times \overline{B}^{X_0}_{\delta/2} \times \overline{B}^{X_0}_{\delta} \to \overline{B}^{X_0}_{\delta}$, which, by the previous discussion is a uniform contraction on the third factor: 
\begin{align*}
    \sup_{(s,y_0,x_0) \in [0,2\tau] \times \overline{B}^{X_0}_{\delta/2} \times \overline{B}^{X_0}_{\delta}} & \Vert D_{x_0} \chi (s,y_0,x_0) \Vert_{\mathcal{L}(X_0)} \\
   & = \sup_{(s,x_0) \in [0,2\tau] \times \overline{B}^{X_0}_{\delta}}  \left\Vert \mathrm{id}_{X_0} - e^{-\widetilde{P}_0 G \imath_0 s} D \eta_s (x_0) \right\Vert_{\mathcal{L}(X_0)} \\
    &\leq \sup_{(s,x_0) \in [0,2\tau] \times \overline{B}^{X_0}_{\delta}}  \left\Vert e^{-\widetilde{P}_0 G \imath_0 s} \right\Vert_{\mathcal{L}(X_0)} \left\Vert  e^{\widetilde{P}_0 G \imath_0 s} -  D \eta_s(x_0)  \right\Vert_{\mathcal{L}(X_0)} < \frac{1}{2}, 
\end{align*}
where the last inequality follows by the choice of $\varepsilon$ and $\delta$.

These two observations ensure that we may apply the uniqueness claim of Theorem~\ref{thm:maps}\ref{thm:maps_2} between two solution pairs of \eqref{eq:sfproof_diag2}, $(g,\pi)$ and $(g,\pi \circ \varphi_s)$, 
uniformly with respect to $s \in [0,2\tau]$, with $U_0 = B_{\delta/3}^{X_0} \cap \widetilde{W}$. 
Hence, we have obtained a neighborhood, say $U \subset \widetilde{W}$, on which 
\begin{equation}
    \pi \circ \varphi_s = \vartheta_s \circ \pi, \qquad s \in [0,2\tau],
    \label{eq:varphi_s_invariance}
\end{equation}
for a   
unique $C^r_b$ diffeomorphism $\vartheta_s$ on $\pi(U)$, which is given explicitly by
\begin{equation}
        \vartheta_s = \pi \circ \varphi_s \vert_{\widetilde{U}_0} \circ ( \pi \vert_{\widetilde{U}_0} )^{-1}, \qquad s \in [0,2\tau],
        \label{eq:theta_explicit}
\end{equation}
on $\pi (U)$, for some  $\widetilde{U}_0 \subset B_{\delta/3}^{X_0}$,
according to the proof of Theorem~\ref{thm:maps}\ref{thm:maps_2}. 
(Note that $\vartheta_\tau =g$ by uniqueness on $\pi(U)$; moreover, $\vartheta_\tau =g$ commutes with $\vartheta_s$, $s \in [0,2\tau]$ by \eqref{eq:g_uniqueness} on $\pi(U)$.) 
For larger times, $\vartheta_s$ will be defined via iterating $g$; we merely use \eqref{eq:varphi_s_invariance} to define $\vartheta_s$ over $[0,\tau]$.
The two definitions of course coincide on overlaps due to uniqueness, but the point in considering \eqref{eq:varphi_s_invariance} over $[0,2\tau]$ is to facilitate the proof of the semiflow property below.

For an arbitrary time $s > 0$, set 
\begin{equation}
    \vartheta_s: = \vartheta_{s-k\tau} \circ g^k = \vartheta_{s-k\tau} \circ \vartheta_\tau^k  
    \label{eq:theta_t_stacking}
\end{equation}
on $\pi(U)$,
where $k$ is the unique integer such that $k\tau \leq s < (k+1) \tau$.
Note \eqref{eq:theta_t_stacking} is well-defined (as it can be ensured, up to possibly shrinking $U$, that $\vartheta_\tau =g$ leaves $\pi(U)$ invariant by Remark~\ref{remark:adapt})  
and a $C^r_b$ diffeomorphism from $\pi(U)$ onto its image.  
Note that, through \eqref{eq:theta_explicit} and \eqref{eq:theta_t_stacking}, $\vartheta$ inherits the smoothness properties of $\varphi \vert_{\mathcal{D}^\varphi \cap (\mathbb{R}^{\geq 0} \times X_0)}$.

Choose $V \subset U$ such that $\varphi_s( V) \subset U$ for all $s \geq 0$ -- this will be the domain we show the semiflow property on. 
Via definition \eqref{eq:theta_t_stacking}, it is clear that
\begin{displaymath}
    \pi \circ \varphi_s = \pi \circ \varphi_\tau^k \circ \varphi_{s-k\tau} = g^k \circ \vartheta_{s-k\tau} \circ \pi = \vartheta_s \circ \pi
\end{displaymath}
on $V$,  
extending \eqref{eq:varphi_s_invariance} to all $s \geq 0$.
For $ 0 \leq t,s  < \tau$, we have
\begin{displaymath}
    \pi \circ \varphi_{t+s} = \vartheta_{t+s} \circ \pi \qquad \text{on } V,
\end{displaymath}
 and 
\begin{displaymath}
    \pi \circ \varphi_{t+s} = \pi \circ \varphi_t \circ \varphi_s = \vartheta_{t} \circ \vartheta_s \circ \pi \qquad \text{on } V,
\end{displaymath}
and hence by the uniqueness property of \eqref{eq:varphi_s_invariance}, $\vartheta_{t+s} = \vartheta_t \circ \vartheta_s$ on $\pi(V)$.
If $t,s \geq \tau$, the same continues to hold via  definition \eqref{eq:theta_t_stacking} and the fact that $g$ and $\vartheta_q$ commute on $\pi(U)$ for $q \leq \tau$.

We have thus far shown that $\vartheta$ defines a local semiflow over $\mathbb{R}^{\geq 0} \times \pi(V)$, satisfying \eqref{eq:sf_foliation_diag}. 
To extend $\vartheta$ to a local flow, set
\begin{equation}
     \vartheta_{-s} : = (\vartheta_s)^{-1}:\vartheta_s(\pi(U)) \to \pi(U), \qquad \text{for } s \geq 0.
     \label{eq:theta_-s}
\end{equation}

There are two cases to consider for the flow property.
In the case $s+t > 0$, $s < 0$, we may assume that $s+t < \tau$ via the semiflow property. 
Then, 
\begin{displaymath}
    \vartheta_{-s} \circ \pi \circ \varphi_{t+s} = \pi \circ \varphi_{-s} \circ \varphi_{t+s} = \pi \circ \varphi_t = \vartheta_t \circ \pi 
\end{displaymath}
holds on $V$, and hence, by uniqueness, $\vartheta_{t+s} = \vartheta_t \circ \vartheta_s$ on $\pi(V)$.
If $s+t < 0$, then 
\begin{displaymath}
    \vartheta_{t+s} = \vartheta_{-(s+t)}^{-1} = (\vartheta_{-s} \circ \vartheta_{-t})^{-1} = \vartheta_t \circ \vartheta_s
\end{displaymath}
on $\vartheta_{-(t+s)}(\pi(V))$.

We have thus shown that on 
\begin{equation}
  \mathcal{D}^\vartheta  = \left( \bigcup_{s < 0} \{ s \} \times \vartheta_{-s}( \pi(V) )\right) \cup \mathbb{R}^{\geq 0} \times \pi(V), 
    \label{eq:thetasetdef}
\end{equation}
$\vartheta$ satisfies the flow property.
The set $\mathcal{D}^\vartheta \cap (\mathbb{R}^{\leq 0} \times X_0)$ is the preimage of $\mathbb{R}^{\geq 0} \times \pi(V)$ under the map 
\begin{gather*}
    \Psi : \left( \bigcup_{s \leq 0} \{ s \} \times \vartheta_{-s}( \pi(U) )\right) \to  \mathbb{R}^{\geq 0} \times \pi(U) \\
    (s,x) \mapsto (-s, \vartheta_s (x)). 
\end{gather*}
If $\varphi \vert_{\mathcal{D}^\varphi \cap (\mathbb{R}^{\geq 0} \times X_0)}$ is jointly continuous, then so is $\Psi$, since the inverse map, $\vartheta_s$ with $s \leq 0$, enjoys the same joint smoothness properties as $\varphi \vert_{\mathcal{D}^\varphi \cap (\mathbb{R}^{\geq 0} \times X_0)}$ through the definition of $\chi$ above and Theorem~21 of \cite{Irwin72}.
Hence, in this case, $\mathcal{D}^\vartheta$ is open. 
As indicated above, joint $C^k$-smoothness of $\varphi \vert_{\mathcal{D}^\varphi \cap (\mathbb{R}^{\geq 0} \times X_0)}$ is inherited by $\vartheta$, $k \leq r$.

If instead, $X_0$ is finite dimensional and $\varphi$ is continuous in time over $X_0$, then, by \eqref{eq:theta_explicit}, $\vartheta$ is separately continuous in time and $C^r$ in space.
It follows, by Theorem~(8A.3) and Remark~(8A.5;4) that $\vartheta$ is jointly continuous.
Joint $C^r$-smoothness of $\vartheta$ then follows from its spatial smoothness by Theorems~(8A.6-7) of \cite{marsden1976hopf} (see in particular Remark~(8A.8;2) thereafter).

The final assertion of part~\ref{thm:sf_1}
follows from the final statements of parts \ref{thm:maps_1} and \ref{thm:maps_2} of Theorem~\ref{thm:maps} (concerning the linear case). 

We prove part~\ref{thm:sf_2}.
Let $(\vartheta,\pi)$ and $(\tilde{\vartheta},\tilde{\pi})$ denote the $C^k_b$  (in space) 
solution pairs as in the statement.
Observe that Remark~\ref{remark:diffeo} continues to hold for the present case, and by considering  
\begin{displaymath}
    \left((D\pi(0) \vert_{X_0})^{-1} \vartheta_t \circ D\pi(0) \vert_{X_0},( D\pi(0) \vert_{X_0})^{-1} \pi \right)
\end{displaymath}
in place of $(\vartheta_t,\pi)$,
we may assume $D\pi(0) \vert_{X_0} = \mathrm{id}_{X_0}$ (and similarly for $\tilde{\pi}$).
Hence, it is possible to construct two diagrams of the form \eqref{eq:diag_sfproof}, one for $(\vartheta_\tau,\pi)$ and one for $(\tilde{\vartheta}_\tau,\tilde{\pi})$.
By Theorem~\ref{thm:maps}\ref{thm:maps_2}, we obtain a local $C^k_b$ diffeomorphism $\theta$ satisfying \eqref{eq:uniqueness_sf} for $t = \tau$ (which is linear if both $\vartheta_\tau$ and $\tilde{\vartheta}_\tau$ are).
The rest follows from \eqref{eq:theta_explicit}, which, for full solutions, holds over all $t \geq 0$ (or $t \in \mathbb{R}$ if $\vartheta$ and $\tilde{\vartheta}$ are flows).   
\end{proof}

\appendix
\section{Proof of Corollary~\ref{cor:Koop_uniq}} 
\label{appendix:B} 
\begin{proof}[Proof of Corollary~\ref{cor:Koop_uniq}]
        Argue by contradiction, and suppose that $\lambda \notin n \sigma(G)$ for all $1 \leq n \leq \ell$.
    Then, by Lemma~\ref{lemma:pkoop_in_spect}, $\psi$ is not a principal eigenfunction, i.e., $D \psi (0) = 0$.
    
    Proceed by induction and assume, for $2 \leq i \leq \ell$, that $D^{j} \psi (0) = 0 $ for $0 \leq j \leq i-1$.
    Differentiating $i$ times the relation $\psi \circ \varphi_\tau = e^{\lambda \tau} \psi$, 
    we obtain
    \begin{displaymath}
        \left( \mathrm{id}_{\mathrm{M}_{i}(X;\mathbb{F})} - l_{e^{- \lambda \tau}} r_{D \varphi_\tau (0)} \right) D^{i} \psi (0) = 0.
    \end{displaymath}
    Using the contradiction hypothesis, we may select $\tau > 0$ in accordance with Remark~\ref{remark:spectral} so that spectral nonresonance holds, implying that $D^{i}\psi (0) = 0$.
    Hence, $\psi \in \Gamma_{r,\ell}$.

    Via spectral mapping (Theorem~\ref{thm:SM}) and up to adapting the norm on $X$ to $D\varphi_\tau(0)$ (Remark~\ref{remark:adapt}), we may assume that $\Vert D\varphi_\tau(0) \Vert_{\mathcal{L}(X)}$ is strictly less than one and is sufficiently close to its spectral radius.
    This implies, along with the assumed Lyapunov stability \ref{hypA1}, that $0$ is asymptotically stable; moreover, via \eqref{eq:Koop_uniq_ass}, that assumption \eqref{eq:ell2} of Lemma~\ref{lemma:S} holds.
    Hence, Lemma~\ref{lemma:S} applied on $\Gamma_{r,\ell}$ to the pair $(f,A_0) =( \varphi_\tau,e^{\lambda \tau})$ with the usual scaling argument (Remark~\ref{remark:scaling}) implies the existence of a $\delta > 0$ such that $\psi \vert_{B_\delta^{X}} \equiv 0$.
    
    We may assume, by asymptotic stability, that $\psi$ is defined on its basin of attraction $\mathcal{B}(0)$ (Remark~\ref{remark:koopman_ext}).  
    The defining property of Koopman eigenfunctions \eqref{eq:Koopman_def} then implies that $\psi(x) = 0$ for any $x \in \mathcal{B}(0)$, upon choosing $t$ large enough so that $\varphi_t(x) \in B_\delta^{X}$.
    We have thus obtained a contradiction, for $\psi \equiv 0$ is not a Koopman eigenfunction according to Definition~\ref{def:koopman}.
\end{proof}

\section{Joint continuity for the Navier-Stokes example} 
\label{appendix:A}

We prove here the claim stated 
during the course 
of the example in Section~\ref{sect:example} concerning the Navier-Stokes system,
namely, the joint continuity of $(t,v_0) \mapsto D^i (\varphi_t  \vert_{V_0})(v_0)$, $i \geq 0$, at $(0,0)$, for $V_0 \subset \mathrm{dom}(A_U)$.

We recall that, by  Theorem 2.1(v) of \cite{WEISSLER1979} (or Proposition~A.1(iii) of \cite{buza2023spectral}, when applied to the Navier-Stokes equations),
one has
for any $u \in V$ and $\tau \in \mathcal{D}^\varphi_{u}$ an open neighborhood $O \subset \mathcal{D}^{\varphi}_\tau$ of $u$ in $V$ such that 
\begin{equation}
    \Vert \varphi_t (u) - \varphi_t (v) \Vert_V \leq C_\tau \Vert u - v \Vert_V, \qquad   0 \leq t \leq \tau,
    \label{eq:NSE_Lip}
\end{equation}
for all $v \in O$.

Note also that the smoothing achieved through the analytic semigroup $t \mapsto e^{A_U t}$ together with the properties of the bilinear operator $B$ yield
\begin{equation}
    \Vert e^{A_U t} B(u,v) \Vert_V \leq C_\alpha t^{-\alpha} e^{\omega t} \Vert u \Vert_V \Vert v \Vert_V, \qquad t > 0, \; u,v \in V,
    \label{eq:AUB_bound}
\end{equation}
for any $\omega > \sup \{ \mathrm{Re} \, \lambda \, | \, \lambda \in \sigma(A_U) \}$ 
and $\alpha \in (3/4,1)$ 
(see (2.12), (2.15) and Lemma~3.1 of \cite{buza2023spectral}).
(In our case, $\omega$ may be taken strictly less than $0$ due to the spectral assumptions.)

\begin{lemma} \label{lemma:Dvarphi_Lip}
Let $\varphi$ denote the Navier-Stokes semiflow on $V$ (as in \eqref{eq:V}).
Fix $u \in V$. 
 Then, there exists $\tau > 0$ and an open neighborhood $O \subset \mathcal{D}^\varphi_{\tau}$ of $u$ in $V$ such that
\begin{equation}
     \Vert D^i \varphi_t (u) - D^i \varphi_t (v) \Vert_{\mathrm{M}_i(V;V)} \leq C_\tau \Vert u - v \Vert_{V}, \qquad 0 \leq t \leq \tau,
    \label{eq:Dvarphi_Lip}
\end{equation}
for all $v \in O$ and $i \geq 0$. 
\end{lemma}

\begin{proof}
The case $i = 0$ is \eqref{eq:NSE_Lip}.

We show the case $i = 1$.  
By continuity in time, we may choose $\tau > 0$ such that 
\begin{equation}
    \Vert \varphi_s(u) \Vert_V < \frac14  \frac{1-\alpha}{C_\alpha \tau^{1-\alpha} \max \{ 1, e^{\omega \tau} \}} , \qquad \text{for all } s \in [0,\tau],
    \label{eq:jointcont_bound}
\end{equation} 
with $\alpha$, $C_\alpha$ and $\omega$ as in \eqref{eq:AUB_bound}.
Set $O \subset \mathcal{D}^{\varphi}_\tau$ to be the neighborhood on which \eqref{eq:NSE_Lip} holds.

The derivative of the semiflow, $D \varphi_t (u)$, may be explicitly written as (see Theorem 2.2 of \cite{WEISSLER1979}), 
\begin{displaymath}
    D \varphi_t (u) [y] = e^{A_U t} y - \int_0^t e^{A_U (t-s)} DB(\varphi_s(u))  D \varphi_s (u) [y] ds, \qquad t \in \mathcal{D}_{u}^\varphi, \; y \in V,
\end{displaymath} 
 where $A_U$ and $B$ are as in \eqref{eq:AU} and \eqref{eq:B}; and $D B (u)[v] = B(u,v) + B(v,u)$. 
We seek to estimate 
\begin{multline*}
     D \varphi_t (u) [y] - D \varphi_t (v) [y]  = - \int_0^t e^{A_U (t-s)} \Big[  DB (\varphi_s(u)) ( D \varphi_s (u) - D \varphi_s (v) ) [y]  \\
     +  DB (\varphi_s(u)-\varphi_s(v)) D \varphi_s (v) [y]  \Big] ds
\end{multline*} 
for $v \in O$ and $y \in V$.
Using \eqref{eq:AUB_bound},
\begin{multline*}
     \left\Vert D \varphi_t (u) [y] - D \varphi_t (v) [y] \right\Vert_V 
     \leq 2 \int_0^t (t-s)^{-\alpha} C_\alpha e^{\omega (t-s)} \\
     \times \Big[ \Vert \varphi_s(u) \Vert_V \Vert ( D \varphi_s (u)  - D \varphi_s (v) ) [y] \Vert_V 
     +  \Vert \varphi_s(u)-\varphi_s(v) \Vert_V \Vert D \varphi_s (v) [y] \Vert_V \Big] ds.
\end{multline*}
Consequently,
\begin{multline*}
    \left\Vert D \varphi_t (u) [y] - D \varphi_t (v) [y] \right\Vert_V \leq 2 C_\alpha \frac{t^{1-\alpha}}{1-\alpha} \max \{ 1, e^{\omega t} \} \\
    \times \sup_{s \in [0,t]}  \big[  \Vert \varphi_s(u) \Vert_V \Vert ( D \varphi_s (u) - D \varphi_s (v) ) [y] 
    +  \Vert \varphi_s(u)-\varphi_s(v) \Vert_V \Vert D \varphi_s (v) [y] \Vert_V \big].
\end{multline*}

Taking the supremum over $t \in [0,\tau]$, using \eqref{eq:jointcont_bound} and \eqref{eq:NSE_Lip}, we obtain
\begin{displaymath}
    \sup_{t \in [0,\tau]} \left\Vert D \varphi_t (u) [y] - D \varphi_t (v) [y] \right\Vert_V \leq  C_{\tau,\alpha } \Vert u - v \Vert_V \sup_{t \in [0,\tau]} \Vert D \varphi_t (v) [y] \Vert_V
    \label{eq:Dvarphit_bound}
\end{displaymath} 
It follows, by the uniform boundedness theorem, that 
\begin{displaymath}
    \sup_{t \in [0,\tau]} \left\Vert D \varphi_t (u) - D \varphi_t (v) \right\Vert_{\mathcal{L}(V)} \leq  C_{\tau,\alpha } \Vert u - v \Vert_V.
\end{displaymath}
This concludes the case $i = 1$;
cases $i \geq 2$ follow trivially, as $D^i \varphi_t(u)$ becomes independent of $u$ for $i \geq 2$. 
\end{proof}

Suppose now that $u = 0$ is a Lyapunov stable fixed point.
Then, there exists a neighborhood $O \subset V$ of $0$ on which the Lipschitz property \eqref{eq:NSE_Lip} continues to hold uniformly as $\tau \to \infty$.
Moreover, \eqref{eq:jointcont_bound} is trivially satisfied for all $\tau \geq 0$.
Via the above proof, we may hence conclude that \eqref{eq:Dvarphi_Lip} also holds for all $\tau \geq 0$ over $O$, which, combined with the boundedness of $D^i \varphi_t(0)$ for each $t \geq 0$ and $i \geq 0$, yields that $\varphi_t$ is $C^\infty_b(O;X)$ for all $t \geq 0$ fixed.

The following Lemma establishes the sought-after joint continuity property.

\begin{lemma}
    Let $\varphi$ denote the Navier-Stokes semiflow on $V$ (as in \eqref{eq:V}), and let $V_0$ be any closed linear subspace such that $V_0 \subset \mathrm{dom}(A_U)$ (with $A_U$ as in \eqref{eq:AU}). 
    Then, the map
    \begin{gather*}
        \mathbb{R}^{\geq} \times V_0 \to \mathrm{M}_i(V_0;V) \\
        (t,v_0) \mapsto D^i (\varphi_t \circ \imath_0) (v_0) 
    \end{gather*} 
    is jointly continuous at $(0,0)$, for all $i \geq 0$ (here $\imath_0: V_0 \xhookrightarrow{} V$ is the inclusion).
\end{lemma}

\begin{proof}
    We first show the case $i = 1$.
    Recall that $A_U$ is a sectorial operator.
    For all $v_0 \in V_0 \subset \mathrm{dom}(A_U)$, we have that
    \begin{displaymath}
        \frac{e^{A_U t}  v_0 - v_0}{t}
    \end{displaymath}
    converges to $A_U v_0$ as $t \to 0$.
    Hence, the uniform boundedness principle implies that, for any $\tau > 0$,
    \begin{displaymath}
        \Vert e^{A_U t} \imath_0 - \imath_0 \Vert_{\mathcal{L}(V_0;V)} \leq C t, \qquad 0 \leq t \leq \tau. 
    \end{displaymath} 

    Supposed we are given $\varepsilon > 0$.
    For $v_0 \in V_0$, we estimate
    \begin{align} 
        \Vert D ( \varphi_t \circ \imath_0) (v_0)  - \imath_0 \Vert_{\mathcal{L}(V_0;V)} &\leq  \Vert D ( \varphi_t \circ \imath_0) (v_0)  - D (\varphi_t \circ \imath_0) (0) \Vert_{\mathcal{L}(V_0;V)} \nonumber \\
       &  \qquad + \Vert D ( \varphi_t \circ \imath_0) (0)  - \imath_0 \Vert_{\mathcal{L}(V_0;V)} \nonumber \\
       &\leq  \Vert D  \varphi_t (v_0) \imath_0  - e^{A_U t} \imath_0 \Vert_{\mathcal{L}(V_0;V)} + C t. \label{eq:applemma2}
    \end{align} 
    Choose $\tau < \varepsilon/(2C)$ and $O \subset \mathcal{D}^\varphi_\tau$ small enough so that \eqref{eq:Dvarphi_Lip} applies. 
    Now taking $ \delta < \varepsilon/(2 C_\tau)$ (with $C_\tau$ as in \eqref{eq:Dvarphi_Lip}, and  such that $B_\delta^{V_0} \subset O$) yields the claim, as 
    $D (\varphi \circ \tilde{\imath}_0) \left( [0,\tau) \times B_\delta^{V_0} \right) \subset B_\varepsilon^{\mathcal{L}(V_0;V)}(\imath_0)$.  

    The case $i = 2$ is treated similarly. 
    By Lemma~\ref{lemma:Dvarphi_Lip} and an argument analogous to \eqref{eq:applemma2}, it suffices to show $t \mapsto D^2 (\varphi_t \circ \imath_0 ) (0)$ is norm-continuous at $t = 0$.
    Observe that 
    \begin{displaymath}
        (t,s) \mapsto e^{A_U (t-s)} B(e^{A_U s}  y_0, e^{A_U s}  \tilde{y}_0)
    \end{displaymath}
    is continuous into $V$ for each fixed pair $y_0,\tilde{y}_0 \in V_0$ and $t \geq s \geq 0$. 
    Hence,
    \begin{multline*}
                \lim_{t \searrow 0} t^{-1} D^2 (\varphi_t \circ \imath_0 )  (0) [y_0,\tilde{y}_0] = \\ -\lim_{t \searrow 0} t^{-1} \int_0^t 2 e^{A_U (t-s)} B(e^{A_U s}  y_0, e^{A_U s} \tilde{y}_0) ds = - 2B(y_0,\tilde{y}_0),
    \end{multline*}
    for all $y_0 ,\tilde{y}_0 \in V_0$, and, by uniform boundedness once more, $t \mapsto D^2 (\varphi_t \circ \imath_0 ) (0)$ is norm-continuous (at $0$).

    The remainder ($i \neq 1,2$) follows trivially from Lemma~\ref{lemma:Dvarphi_Lip}, since $D^i \varphi_t (0) = 0$ for all $t \geq 0$ and $i \neq 1,2$.
\end{proof}

\section*{Acknowledgments}
The author gratefully acknowledges the support of the Harding Foundation through a PhD scholarship (https://www.hardingscholars.fund.cam.ac.uk).

\bibliographystyle{alpha}
\bibliography{koopmans}

\end{document}